\documentclass[12pt]{amsart}

\usepackage[utf8]{inputenc}
\usepackage{graphicx}
\usepackage{amssymb}
\usepackage{amsfonts}
\usepackage{amsmath,amscd}
\usepackage{array}
\usepackage{rotating}
\usepackage{epsfig}
\usepackage[colorlinks,citecolor=blue]{hyperref}
\usepackage{tikz}
\usetikzlibrary{matrix}
\usepgflibrary{arrows}

\def\qandq{\qquad\text{and}\qquad}

\newtheorem{example}{Example}[section]
\newtheorem{note}[example]{Note}
\newtheorem{theorem}[example]{Theorem}

\newtheorem{corollary}[example]{Corollary}

\newtheorem{definition}[example]{Definition}
\newtheorem{proposition}[example]{Proposition}

\newtheorem{lemma}[example]{Lemma}

\def\Proof{\noindent \it Proof -- \rm}
\def\qed{\hspace{3.5mm} \hfill \vbox{\hrule height 3pt depth 2 pt width 2mm}
\bigskip}

\def\SetSp{\mathcal{S}et}
\def\FSp{\mathcal{F}}
\def\ZSp{\mathcal{Z}}
\def\BSp{\mathcal{R}}
\def\RSp{\mathcal{R}}
\def\WSp{\mathcal{W}}

\def\<{\langle}
\def\>{\rangle}
\def\NN{{\mathbb N}}    
\def\QQ{{\mathbb Q}\, } 

\def\G{{\bf G}}         

\def\gautrid{\!\prec\!}   
\def\miltrid{\circ}       
\def\droittrid{\!\succ\!} 

\def\BWU{\mathbb{RW}} 

\def\BWTSL{{\rm RW^\Sigma}} 
\def\BWTL{{\rm RW^\Sigma_{T}}} 
\def\BWSL{{\rm RW^\Sigma_{DS}}} 
\def\BWL{{\rm RW^\Sigma_{D}}} 

\def\BWTS{{\rm RW}} 
\def\BWT{{\rm RW_{T}}} 
\def\BWS{{\rm RW_{DS}}} 
\def\BW{{\rm RW_{D}}} 

\newcommand\opMould[1]{\mathcal{M}\mathrm{ould}^{#1}}
\def\opGR{\mathcal{GR}_4}
\def\opFree{\mathcal{G}_4}

\def\opFF{\mathcal{F\!F}} 
\def\opMon{\mathcal{M}\mathrm{on}} 
\def\opMFF{\mathcal{MF\!F}} 
\def\NAP{\operatorname{NAP}}
\newcommand\opRatFct[1]{#1\mathrm{\text{-}\mathcal{R}at\mathcal{F}ct}}
\newcommand\opDend{\mathrm{\mathcal{D}end}}
\newcommand\opTriDend{\mathrm{\mathcal{T}ridend}}
\newcommand\SET[1]{\operatorname{SET}(#1)}
\newcommand\SYMOP[1]{{#1}^\Sigma}

\tikzset{baseline={([yshift=-3.5pt]current bounding box.center)}}
\tikzset{level distance = 0.7cm, sibling distance = 2em}
\tikzset{edge from parent/.style={
       draw,
       edge from parent path = {(\tikzparentnode) -- 
                                (\tikzchildnode)} 
    }}
\tikzset{root/.style={inner sep=1.5pt, minimum width=1.2em,
                      fill=blue!30, rounded corners}}
\tikzset{white/.style={inner sep=1.2pt, minimum width=1.2em,
                       fill=blue!10, rounded corners}}
\tikzset{red/.style={inner sep=1.2pt, minimum width=1.2em,
                       fill=red, rounded corners}}

\def\treeone#1{\begin{tikzpicture}
\node [root] {$#1$}
  child [missing]
  child [missing];
\end{tikzpicture}}

\def\geno#1#2{\begin{tikzpicture}
\node [root]{$\scriptstyle \left\{#1, #2\right\}$};
\end{tikzpicture}}
\def\geng#1#2{\begin{tikzpicture}
\node [root]{$\scriptstyle \left\{#1\right\}$}
  child {
    node [white] {$\scriptstyle \left\{#2\right\}$}
  };
\end{tikzpicture}}
\def\gend#1#2{\begin{tikzpicture}
\node [root]{$\scriptstyle \left\{#2\right\}$}
  child {
    node [white] {$\scriptstyle \left\{#1\right\}$}
  };
\end{tikzpicture}}
\def\gens#1#2{\begin{tikzpicture}
\node [fill=blue!30] [red] {$\scriptstyle \left\{\right\}$}
  child {
    node [white] {$\scriptstyle \left\{#2\right\}$}
  }
  child {
    node [white] {$\scriptstyle \left\{#1\right\}$}
  };
\end{tikzpicture}}

\def\treeNode#1{\begin{tikzpicture}
\node [root] {$#1$};
\end{tikzpicture}}
\def\treeLft#1#2{\begin{tikzpicture}
\node [root] {$#1$}
  child {
    node [white] {$#2$}
      child [missing]
      child [missing]
  }
  child [missing];
\end{tikzpicture}}
\def\treeRgt#1#2{\begin{tikzpicture}
\node [root]{$#1$}
  child [missing]
  child {
    node [white] {$#2$}
      child [missing]
      child [missing]
  };
\end{tikzpicture}}

\title[]%
{A set-operad of formal fractions\\
and dendriform-like sub-operads}

\author[F. Chapoton, F. Hivert, and J.-C.~Novelli]%
{Fr\'ed\'eric Chapoton, Florent Hivert, and Jean-Christophe Novelli}

\address[Chapoton]{Institut Camille Jordan, Universit\'e Claude Bernard Lyon 1,
69622 Villeurbanne Cedex, FRANCE}
\address[Hivert]{Laboratoire de Recherche en Informatique, Universit\'e
Paris-Sud, 91405 Orsay Cedex, FRANCE}
\address[Novelli] {Institut Gaspard Monge, Universit\'e de Marne-la-Vall\'ee \\
5 Boulevard Descartes \\Champs-sur-Marne \\77454 Marne-la-Vall\'ee cedex 2 \\
FRANCE}
\email[Fr\'ed\'eric Chapoton]{chapoton@math.univ-lyon1.fr}
\email[Florent Hivert]{hivert@lri.fr}
\email[Jean-Christophe Novelli]{novelli@univ-mlv.fr}
\date{\today}

\begin{document}

\begin{abstract}
  We introduce an operad of formal fractions, abstracted from the
  Mould operads and containing both the Dendriform and the
  Tridendriform operads. We consider the smallest set-operad contained
  in this operad and containing four specific elements of arity two,
  corresponding to the generators and the associative elements of the
  Dendriform and Tridendriform operads. We obtain a presentation of
  this operad (by binary generators and quadratic relations) and an
  explicit combinatorial description using a new kind of bi-colored
  trees. Similar results are also presented for related symmetric operads.
\end{abstract}

\maketitle

{\footnotesize
\tableofcontents
}

The main theme of this article is about combinatorial and algebraic
descriptions of some set-operads. The notion of operad has its historical
roots in algebraic topology, and has become a useful and classical tool in
this field. More recently, operads have also been considered from a more
algebraic point of view, namely in the monoidal categories of vector spaces
and chain complexes instead of the monoidal category of topological
spaces. The homology functor is a natural way to pass from topological operads
to algebraic operads.

But operads can also be considered with a combinatorial state of mind,
and the natural ambient category is then the monoidal category of
finite sets. If one is given an operad $P$ in the category of vector
spaces, there is a simple idea to obtain an operad in the category of
finite sets: choose a finite set of elements of $P$ and consider the
closure of this set under the composition maps of $P$. One can then
try to count the finite sets obtained in this way, and to describe
their elements in an explicit way.

As a side remark, let us note that there is an algebraic motivation
for doing this, related to categorification. If the underlying vector
spaces of an operad could be considered as the Grothendieck groups of
some Abelian categories, and composition maps as coming from functors
between these categories, then elements of the operad would correspond
to objects of these categories. Finding subsets of elements closed
under the composition maps and describing their combinatorics could be
a way to find hints on the nature of objects in the Abelian
categories.
\bigskip

This article started with the aim to apply this closure procedure to
the generators and their sums, in two operads in the category of
vector spaces introduced by J.-L. Loday, namely the dendriform
\cite{LR} and tridendriform operads \cite{LRtri}. It has been proved
in \cite{Ch} and \cite{CHNT} (see also \cite{Lod10}) that these
operads can be considered as sub-operads of two different operads of
fractions. Our problem is therefore to describe combinatorially the fractions that can
be obtained by iterated compositions of the fractions corresponding to
generators of $\opDend$ or $\opTriDend$ and their sums.

Because these two operads of fractions have very similar composition
maps, one can define a set-operad of formal fractions $\opFF$ in which
the closure problems for $\opDend$ and $\opTriDend$  can be considered
simultaneously. Indeed, the chosen subset of $\opDend$ is contained in the
chosen subset of $\opTriDend$ , when both are considered as formal
fractions. One is therefore lead to the following question: describe
the closure of four fractions in $\opFF(2)$ (associated with three
generators of $\opTriDend$  and their sum) under the compositions of the
operad $\opFF$. This defines a set-operad, denoted by $\opFF_4$.

Our main results are a presentation by binary generators and quadratic
relations and an explicit combinatorial description of $\opFF_4$ using a
new kind of bicolored trees, called the red and white trees.

The main interest of those trees is to provide simple ways to answer two
natural questions: compute the automorphism group of a given element
and check if a given tree is in a given set-operad. Indeed, on the description
of an element as a composition of generators, none of these questions can
easily be answered, and on the description as a fraction, only the
automorphism group can effortlessly be seen.

We also consider the similar closure properties for symmetric operads
(with actions of the symmetric groups) and obtain a symmetric analog
of the isomorphism between $\opFF_4$ and the operad of red and white
trees.
\bigskip

The article is organized as follows:

In Section 1, we briefly recall general facts about operads,
dendriform and tridendriform algebras.

In section 2, we recall two known operads on fractions, introduce the
operad of formal fractions, and describe two inclusions of formal
fractions in fractions.

In section 3, we describe the images of the chosen elements of the
Dendriform and Tridendriform operads in formal fractions and define
the operad $\opFF_4$ as the closure of these images. We then introduce
an operad $\opGR$ given by generators and relations, and proceed to
prove that it is isomorphic to $\opFF_4$, using rewriting techniques.

In section 4, we introduce an operad $\BWTS$ on the sets of red and
white trees, its composition maps being given by combinatorial
rules. We prove that this operad is isomorphic to the operad $\opGR$
by comparing their generating series. We then prove the main theorem,
which states that all three operads $\opFF_4$, $\opGR$ and $\BWTS$ are
isomorphic.

In section 5, we use the previous construction to consider various
sub-operads generated by some subsets of the four chosen generators.

In section 6, we extend some of the previous results to the closure as
symmetric operads (instead of non-symmetric operads). In particular,
we obtain a symmetric analog of the isomorphism between $\opFF_4$ and
$\BWTS$.

In section 7, we sketch, mostly without proofs, an extension of all
this work to a set-operad on $6$ generators inside a more general kind
of formal fractions and its relation with a more general kind of red
and white trees.

\subsubsection*{Acknowledgment} This research was partially supported by projet
ANR-12-BS01-0017. The authors thank the Centro di Giorgi (Pise) for its
hospitality. This research was driven by computer exploration, using the
open-source mathematical software \texttt{Sage}~\cite{sage} and its algebraic
combinatorics features developed by the \texttt{Sage-Combinat}
community~\cite{Sage-Combinat}.

\section{Background}

\subsection{Operads}
\label{sec.operads}

We will consider in this article various kinds of operads. Let us fix
our terminology.

First, we will use the word \textit{operad} to mean a non-symmetric
operad, and otherwise talk of \textit{symmetric operad}.

An operad $P$ in a monoidal category with tensor product $\otimes$ is
a collection of objects $P(n)$ for integers $n \geq 1$, endowed with
composition maps $\circ_i$ from $P(m)\otimes P(n) \to P(m+n-1)$ for
all integers $m, n \geq 1$ and $ 1 \leq i \leq m$ satisfying
appropriate associativity axioms. One also requires a unit in
$P(1)$. The detailed definition can be found in many references, for
example \cite{}.

Symmetric operads are slightly more complex structures. A symmetric
operad can be defined as a collection $P(n)$ with an action of the
symmetric group $S_n$ on $P(n)$, these actions being moreover
compatible in the appropriate sense with the composition maps. An
alternative definition can be given using the language of species
\cite{Specie}: a symmetric operad is then a species $P$ and natural
composition maps $\circ_i$ from $P(I)\otimes P(J) \to P(I \setminus
\{i\} \sqcup J)$ for all finite sets $I$ and $J$ and every element $i
\in I$.

Almost all operads that will be considered are operads in the category
of sets endowed with the cartesian product. They will sometimes be
called \textit{set-operads} to avoid ambiguity.

\subsection{The dendriform and tridendriform operads}
\label{sec.dend.tridend}

The notion of a dendriform algebra has been introduced by Loday, in a
sequence of articles involving several other new kinds of algebras,
including Leibniz algebras. A dendriform algebra is an associative
algebra where the associative product $\odot$ can be written as a sum
of two bilinear operations:
\begin{equation}
  x \odot y = x \prec y + x \succ y,
\end{equation}
in such a way that $\prec$ and $\succ$ satisfy three
axioms. Conversely, these three axioms on the operations $\prec$ and
$\succ$ imply that the $\odot$ product is associative. Loday has
described the free dendriform algebras, and therefore the dendriform
operad, using planar binary trees. For more details, the reader may
consult \cite{LR}.

The notion of tridendriform algebra is a variation on the same idea,
where the associative product is cut into three pieces
\begin{equation}
  x \odot y = x \prec y + x \miltrid y + x \succ y,
\end{equation}
in such a way that $\miltrid$ is associative and $\prec, \miltrid$ and
$\succ$ satisfy $6$ other axioms. The free algebras are then described
by planar trees instead of planar binary trees. For more details, see
for example \cite{LRtri}.

\section{The operads of formal fractions}

Inspired by \'Ecalle's mould calculus \cite{eca2002}, the first author
defined in~\cite{Ch} an operad structure $\opMould0$ on the vector
spaces
\begin{equation}
  \opMould0(n) := \QQ(u_1, \dots, u_n)
\end{equation}
of rational fractions in the variables $\{u_1, \dots, u_n\}$.

The composition is defined for $F\in\opMould0(m)$ and $G\in\opMould0(n)$ by
\begin{equation}
\label{def.composition.mould0}
  F \circ_i G := S_{i,n}\,
  F(u_1, \dots, u_{i-1}, S_{i,n}, u_{i+1},\dots,u_{m+n-1})\,
  G(u_i, \dots, u_{i+n-1})\,
\end{equation}
where $S_{i,n}=u_i+u_{i+1}+\dots +u_{i+n-1}$. It was proved
in~\cite{Ch} that the sub-operad of $\opMould0$ generated by the
fractions $\frac{1}{u_1(u_1+u_2)}$ and $\frac{1}{u_2(u_1+u_2)}$ is
isomorphic to the dendriform operad.

The natural action on the symmetric groups endows $\opMould0$ with a
symmetric operad structure. The symmetric sub-operad generated by
$\frac{1}{u_1(u_1+u_2)}$ is isomorphic to the Zinbiel
operad~\cite{CHNT}.

A very similar operad called $\opMould1$, over the same vector spaces,
has been defined in ~\cite{MNT}. In $\opMould1$, the composition is
defined by
\begin{equation}
\label{def.composition.mould1}
  F \circ_i G := (P_{i,n} - 1)\,
  F(u_1, \dots, u_{i-1}, P_{i,n}, u_{i+1},\dots,u_{m+n-1})\,
  G(u_i, \dots, u_{i+n-1})\,
\end{equation}
where $P_{i,n}=u_i u_{i+1}\dots u_{i+n-1}$. The operad $\opMould1$
contains the tridendriform operad, as the sub-operad generated by
\begin{equation}
\frac{1}{(u_1-1)(u_1u_2-1)},\quad \frac{1}{(u_2-1)(u_1u_2-1)},
\quad \text{and\ } \frac{1}{u_1u_2-1}.
\end{equation}
It is also a symmetric operad and its symmetric sub-operad generated
by $\frac{1}{(u_1-1)(u_1u_2-1)}$ and $\frac{1}{u_1u_2-1}$ is called
CTD\footnote{standing for Commutative TriDendriform} by Loday \cite{LodQS}.

Interestingly, these two operads are particular cases of a family of operads
indexed by a parameter $\lambda$ defined by Loday in \cite{Lod10} and denoted
by $\opRatFct{\lambda}$. He showed that $\opMould0$ and $\opRatFct{0}$ are
isomorphic. In \cite{MNT}, it is showed that $\opMould1$ and $\opRatFct{1}$
are isomorphic too.

\subsection{Formal fractions}

Let us denote by $\SET{\mathcal{O}}$, the set-operad obtained from an
operad $\mathcal{O}$ by forgetting about its linear structure. In the
present paper, we deal with some sub-operads of $\SET{\opMould0}$ and
$\SET{\opMould1}$. It will be handy, as an intermediate tool, to
encode all computations by means of a common sub-operad of both,
namely the operad of formal fractions $\opFF$.

\medskip
Let $\opFF(n)$ be the set of fractions whose numerator and denominator are
products of formal symbols $[S]$ where $S$ is any non-empty subset of
$\{1,\dots,n\}$. As usual, we simplify the fraction if the same symbol appears
on top and bottom. For readability, $[\{1,3,5,6\}]$ is written $[1356]$. The
empty product is denoted by $1$, which should not be confused with the symbol
$[1]$.

For $F\in\opFF(n)$ and $i_1,\dots,i_n$ integers, we write $F(i_1,\dots,i_n)$
the fraction where $k$ is replaced by $i_k$. We extend naturally the
definition to the case where $i_k$ is itself a set by taking union. For
example,
\begin{equation}
\frac{[1][34]}{[13][2]}(2,5,6,9) = \frac{[2][69]}{[26][5]}
\qandq
\frac{[1][34]}{[13][2]}(2,5,\{6,8\},9) = \frac{[2][689]}{[268][5]}\,.
\end{equation}

Let $F\in\opFF(m)$ and $G\in\opFF(n)$. Define composition
$F\circ_i G\in\opFF(m+n-1)$ by
\begin{equation}
\label{def.composition.ff}
  F\circ_i G := [S_{i,n}]\,
  F(u_1, \dots, x_{i-1}, S_{i,n}, u_{i+1},\dots,u_{m+n-1})\,
  G(u_i, \dots, u_{i+n-1})\,
\end{equation}
where $S_{i,n} := \{i,i+1,\dots i+n-1\}$.

For example, using
\begin{gather}
  \frac{[123]}{[1234][12][2][3]}(1,\{2,3,4,5\},6,7) =
  \frac{[123456]}{[1234567][12345][2345][6]} \,,\\
  \frac{1}{[12][34]}(2,3,4,5) = \frac{1}{[23][45]}\,, \\
\intertext{one finds that}
  \frac{[123]}{[1234][12][2][3]} \circ_2 \frac{1}{[12][34]} =
  \frac{[123456]}{[1234567][12345][6][23][45]}\,.\\
\intertext{Similarly,}
  \frac{[123]}{[1234][12][2][3]} \circ_1 \frac{1}{[12][34]} =
  \frac{[123456][1234]}{[1234567][12345][5][6][12][34]}\,.
\end{gather}

\begin{proposition}
  \begin{enumerate}
  \item The family $\opFF := (\opFF(n))_{n\in\NN}$ together with the
    compositions $\circ_i$ is a set-operad.

  \item The map $\phi^0$ sending $[S]$ to $\sum_{i\in S} u_i$ and
    formal fractions to fractions is an injective morphism of
    set-operads from $\opFF$ to $\SET{\opMould0}$.

  \item The map $\phi^1$ sending $[S]$ to $\left(\prod_{i\in S}
      u_i\right)-1$ and formal fractions to fractions is an injective
    morphism of set-operads from $\opFF$ to $\SET{\opMould1}$.
  \end{enumerate}
\end{proposition}

\begin{proof}
  It is clear from the definitions of compositions in Equations
  \eqref{def.composition.mould0}, \eqref{def.composition.mould1} and
  \eqref{def.composition.ff} that both maps $\phi^0$ and $\phi^1$ are
  injective and commute with all compositions. Therefore $\opFF$ itself is an
  operad.
\end{proof}

The symmetric groups act naturally on $\opMould0$, $\opMould1$ and
$\opFF$, endowing these with symmetric operad structures. We shall
denote these operads by adding a $\Sigma$ as exponent.
\begin{proposition}
  The maps $\phi^0$ and $\phi^1$ are respectively injective morphisms from
  $\SYMOP\opFF$ to $\SYMOP{\opMould0}$ and $\SYMOP{\opMould1}$.
\end{proposition}

\section{Dendriform and tridendriform operads in formal fractions}

Recall that according to~\cite{Ch}, the map $\succ\ \mapsto\ \frac{1}{u_1(u_1+u_2)}$
and $\prec\ \mapsto\ \frac{1}{u_2(u_1+u_2)}$ is an injective morphism from the
dendriform operad to $\opMould0$. Note that by definition
\begin{equation}
\phi_0\left(\frac{1}{[1][12]}\right)=\frac{1}{u_1(u_1+u_2)}
\qandq
\phi_0\left(\frac{1}{[2][12]}\right)=\frac{1}{u_2(u_1+u_2)}\,.
\end{equation}

Therefore, the sub-operad of $\SET\opDend$ generated by $\succ$ and
$\prec$ is isomorphic to the sub-operad of $\opFF$ generated by
$\frac{1}{[1][12]}$ and $\frac{1}{[2][12]}$. Moreover, the associative product
$\prec + \succ$ of $\opDend$ is associated to
\begin{equation}
  \frac{1}{u_1(u_1+u_2)} + \frac{1}{u_2(u_1+u_2)} = \frac{1}{u_1u_2}
  = \phi_0\left(\frac{1}{[1][2]}\right)\,,
\end{equation}
and therefore also lives in $\opFF$.

In a similar way, according to~\cite{MNT}, the tridendriform operad is a
sub-operad of $\opMould1$ via the map
\begin{alignat}{2}
\label{def_phi1_sup}
\succ    \ \mapsto\ &\frac{1}{(u_1-1)(u_1u_2-1)} &&= \phi_1\left(\frac{1}{[1][12]}\right)\,,\\
\label{def_phi1_inf}
\prec    \ \mapsto\ &\frac{1}{(u_2-1)(u_1u_2-1)} &&= \phi_1\left(\frac{1}{[2][12]}\right)\,,\\
\label{def_phi1_circ}
\circ\ \mapsto\ &\frac{1}{(u_1u_2-1)}    &&= \phi_1\left(\frac{1}{[12]}\right)\,.\\
\intertext{Moreover}
\label{def_phi1_sum}
(\prec + \circ + \succ)\ \mapsto\ &\frac{1}{(u_1-1)(u_2-1)}    &&= \phi_1\left(\frac{1}{[1][2]}\right)\,.
\end{alignat}
So we can study both the dendriform and tridendriform cases by studying the
case of formal fractions:
\begin{definition}
  We denote by $\opFF_4$ the sub-set-operad of $\opFF$ generated by
  the fractions
  \begin{equation}
    F_{\droittrid}:=\frac{1}{[1][12]},\quad
    F_{\gautrid}:=\frac{1}{[2][12]},\quad
    F_{\miltrid}:=\frac{1}{[12]},\quad
    F_{\odot}:=\frac{1}{[1][2]}\,.
  \end{equation}
\end{definition}

The main goal of this paper is to understand $\opFF_4$ and several of its
sub-operads.

\subsection{Generators and relations}

The first problem is to determine the relations between the four generators.
As we shall see, it turns out that the relations are all in degree $2$.
They are
\begin{equation}
\begin{array}{>{\displaystyle}r@{\ =\ }>{\displaystyle}c@{\ =\ }>{\displaystyle}l}
\frac{1}{[12] [2]} \circ_{1} \frac{1}{[12] [1]} & \frac{1}{[123] [1] [3]} & \frac{1}{[12] [1]} \circ_{2} \frac{1}{[12] [2]}\\[4mm]
\frac{1}{[12]} \circ_{1} \frac{1}{[12]} & \frac{1}{[123]} & \frac{1}{[12]} \circ_{2} \frac{1}{[12]}\\[4mm]
\frac{1}{[12]} \circ_{1} \frac{1}{[12] [1]} & \frac{1}{[123] [1]} & \frac{1}{[12] [1]} \circ_{2} \frac{1}{[12]}\\[4mm]
\frac{1}{[12]} \circ_{1} \frac{1}{[12] [2]} & \frac{1}{[123] [2]} & \frac{1}{[12]} \circ_{2} \frac{1}{[12] [1]}\\[4mm]
\frac{1}{[12] [2]} \circ_{1} \frac{1}{[12]} & \frac{1}{[123] [3]} & \frac{1}{[12]} \circ_{2} \frac{1}{[12] [2]}\\[4mm]
\frac{1}{[1] [2]} \circ_{1} \frac{1}{[1] [2]} & \frac{1}{[1] [2] [3]} & \frac{1}{[1] [2]} \circ_{2} \frac{1}{[1] [2]}\\[4mm]
\frac{1}{[12] [2]} \circ_{1} \frac{1}{[12] [2]} & \frac{1}{[123] [2] [3]} & \frac{1}{[12] [2]} \circ_{2} \frac{1}{[1] [2]}\\[4mm]
\frac{1}{[12] [1]} \circ_{1} \frac{1}{[1] [2]} & \frac{1}{[123] [1] [2]} & \frac{1}{[12] [1]} \circ_{2} \frac{1}{[12] [1]}\\[4mm]
\end{array}
\end{equation}
Note that they correspond to the tridendriform relations:
\begin{equation}
  \begin{aligned}
    (x\droittrid y)\gautrid z &= x\droittrid (y\gautrid z)\,,\\
    (x\miltrid y)\miltrid z &= x\miltrid (y\miltrid z)\,,\\
    (x\droittrid y)\miltrid z &= x\droittrid (y\miltrid z)\,,\\
    (x\gautrid y)\miltrid z &= x\miltrid (y\droittrid z)\,,\\
    (x\miltrid y)\gautrid z &= x\miltrid (y\gautrid z)\,,\\
    (x\odot y)\odot z &= x\odot (y\odot z)\,,\\
    (x\gautrid y)\gautrid z &= x\gautrid (y\odot z)\,,\\
    (x\odot y)\droittrid z &= x\droittrid (y\droittrid z)\,.
  \end{aligned}
\end{equation}
To show that these are the only relations of $\opFF_4$, we need to
consider the quotient of the free set-operad on
$\{\gautrid,\, \droittrid,\, \miltrid,\, \odot\}$ by these relations. Recall
that the free operad on a set $G$ of binary generators is the set of
binary trees with nodes labelled by the elements of $G$. To simplify
the notations, when drawing a tree we never write the leaves of the
tree: for example, the following trees are equal and we shall draw
the first one in the rest of the paper:

\begin{equation}
\prec\ \circ_1\ \odot
\ =\ 
\begin{tikzpicture}
\node [root] {$\prec$}
  child {
    node [white] {$\odot$}
      child [missing]
      child [missing]
  }
  child [missing];
\end{tikzpicture}
\ =\ 
\begin{tikzpicture}
\node [root] {$\prec$}
  child {
    node [white] {$\odot$}
      child {}
      child {}
  }
  child {};
\end{tikzpicture}
\end{equation}

\begin{definition}
  Let us denote by $\opGR$ the quotient of the free set-operad $\opFree$
  generated by $\{\treeNode{\gautrid}$, $\treeNode{\droittrid}$,
  $\treeNode{\miltrid}$, and $\treeNode{\odot}\}$ by the relations:
  \begin{equation}
    \label{eq.def.GR}
    \begin{aligned}
      \treeLft{\prec}{\succ}\ &=\ \treeRgt{\succ}{\prec}
      \\[5mm]
      \treeLft{\circ}{\circ}\ &=\ \treeRgt{\circ}{\circ}
      &
      \treeLft{\circ}{\succ}\ &=\ \treeRgt{\succ}{\circ}
      &
      \treeLft{\circ}{\prec}\ &=\ \treeRgt{\circ}{\succ}
      &
      \treeLft{\prec}{\circ}\ &=\ \treeRgt{\circ}{\prec}
      \\[5mm]
      \treeLft{\odot}{\odot}\ &=\ \treeRgt{\odot}{\odot}
      &
      \treeLft{\prec}{\prec}\ &=\ \treeRgt{\prec}{\odot}
      &
      \treeLft{\succ}{\odot}\ &=\ \treeRgt{\succ}{\succ}
    \end{aligned}
  \end{equation}
\end{definition}
The next paragraphs are devoted to the proof that $\opGR$ is
isomorphic to $\opFF_4$.

\subsection{Canonical trees}
\label{subsection:canonical}

Let us begin our study of $\opGR$ by computing the generating series of its
cardinalities. This is done using rewriting theory.

We start by choosing a tree in each equivalence class modulo the relation:
\begin{definition}
  We say that a tree in $\opFree$ is \emph{canonical} if it avoid all
  patterns on the right-hand side of each of the
  Relations~\eqref{eq.def.GR}.
\end{definition}

\begin{lemma}
  Each equivalence class modulo Relations~\eqref{eq.def.GR} contains at least 
  one canonical tree.
\end{lemma}

\begin{proof}
  Consider a tree $T$.
  We want to show that there is a canonical tree in the class of $T$.

  If it avoids all patterns, the property is true; otherwise,
  replace any pattern by its (left-hand side) image. If one denotes by
  $f$ the function that associates with a tree the sum of the cardinality of
  the right sub-tree of each node, then $f$ strictly decreases from $T$ to the
  new tree $T'$, hence proving that the process stops after a certain number
  of steps.
\end{proof}

Note that $f$ is the classical invariant that shows that the Tamari order is
anti-symmetric.
\medskip

\begin{definition}
  We say that an edge of $T$ is \emph{rewritable} if its extremities belong to
  one of the relations above, either on the left or on the right.
\end{definition}

So if one changes a rewritable edge into its rewriting thanks to the
relations, it remains rewritable by definition. Then one can check that for
all $5$ shapes of binary trees with $3$ nodes, there exists exactly $16$ trees
whose two edges are rewritable. Using the rewriting relations, we can group
those trees into $16$ pentagons of equivalent trees where each shape appears
exactly once and each edge is rewritable. Here is an example:
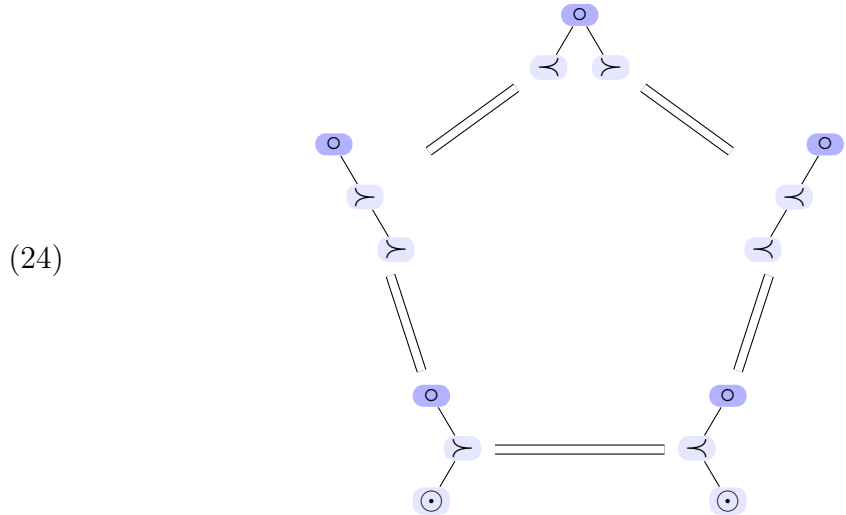
\begin{equation}
\begin{tikzpicture}
\node at (378:3) (a) [inner sep=5pt, fill=white, rounded corners]{
\begin{tikzpicture}
\node [root]{$\circ$}
  child {
    node [white] {$\prec$}
    child {
      node [white] {$\prec$}
    }
    child [missing]
  }
  child [missing];
\end{tikzpicture}};
\node at (90:3) (b) [inner sep=5pt, fill=white, rounded corners] {
\begin{tikzpicture}
  \node [root] {$\circ$}
  child {node [white] {$\prec$}}
  child {node [white] {$\succ$}};
\end{tikzpicture}};
\node at (162:3) (c) [inner sep=5pt, fill=white, rounded corners] {
\begin{tikzpicture}
  \node [root]{$\circ$}
  child [missing]
  child {
    node [white] {$\succ$}
    child [missing]
    child {
      node [white] {$\succ$}
    }
  };
\end{tikzpicture}};
\node at (234:3) (d) [inner sep=5pt, fill=white, rounded corners] {
\begin{tikzpicture}
  \node [root]{$\circ$}
  child [missing]
  child {
    node [white] {$\succ$}
    child {
      node [white] {$\odot$}
    }
    child [missing]
  };
\end{tikzpicture}};
\node at (306:3) (e) [inner sep=5pt, fill=white, rounded corners] {
\begin{tikzpicture}
  \node [root]{$\circ$}
  child {
    node [white] {$\prec$}
    child [missing]
    child {
      node [white] {$\odot$}
    }
  }
  child [missing];
\end{tikzpicture}};
\foreach \from/\to in {a/b,b/c,c/d,d/e,e/a}
    \draw [double distance = 3pt, line cap=none] (\from) -- (\to);
\end{tikzpicture}
\end{equation}

Note that each pentagon contains exactly one canonical tree. As a
consequence, there are only three kinds of equivalence classes of
trees with three nodes: singleton (tree with no rewriting), pairs
(trees with only one rewriting), and pentagons (trees with two
rewritings). This has the fundamental consequence that rewriting an
edge in any tree does not create new rewritable edges. We can therefore
prove that:
\begin{lemma}
  Each equivalence class modulo Relations~\eqref{eq.def.GR} contains exactly
  one canonical tree.
\end{lemma}

\begin{proof}
  We claim that given tree $T$, the number of rewritable edges does not change
  if one rewrites a given edge. It is obvious for a rewritable edge that has
  no node in common with any other rewritable edge. Thanks to the pentagonal
  relations, this is also true if one rewrites an edge adjacent to another
  one. Now, if two rewritable edges do not share a node, one can rewrite both
  in any order and obtain the same final tree.
  If two rewritable edges share a node, the pentagonal relations show that if
  one rewrites any of the two edges, one will always end ultimately with the
  same tree. Otherwise said, the rewriting system is \emph{locally confluent}
  and hence confluent.

  Hence, one can do the rewritings in any order, the resulting tree is always
  the same. So there only is one tree that avoids all patterns on the
  right-hand sides in each class.
\end{proof}

\begin{note}
  Note that two different trees of the same shape are necessarily in
  different classes. Otherwise, one could find two sequences of
  rewritings starting from these trees and leading to the same
  canonical form. Using the pentagons (and squares), one can assume
  that the underlying sequences of rotations of trees are the
  same. But then the sequences of rewritings are the same, hence the
  trees are the same, which is absurd.

  Moreover, all shapes of trees appearing in a class form an interval
  of the Tamari lattice (thanks to the observation about rewritable
  edges) and the representative of a class is the tree closest to the
  left comb.
\end{note}




Finally, note that thanks to the existence of simple canonical trees,
one easily obtains an equation satisfied by the generating series of
the cardinalities of the operad.  Indeed, the eight relations
\eqref{eq.def.GR} when oriented provide the following relations. Let
us denote by $F$, $l$, $m$, $r$, $s$ respectively the generating
series of all canonical trees, all canonical trees having $\prec$,
$\circ$, $\succ$, and $\odot$ as their root.

We then get
\begin{equation}
  \label{systeme_gr4}
  \left\{
    \begin{aligned}
      F &= l + m + r + s + 1, \\
      l &= x \,F\,(l+m+r+1),  \\
      m &= x \,F\,(s+1),      \\
      r &= x \,F\,(s+1),      \\
      s &= x \,F\,(l+m+r+1).  \\
    \end{aligned}
  \right.
\end{equation}

Then, summing the last four equations, one finds
\begin{equation}
  F-1 = x\,F\,2(F+1),
\end{equation}
which is also
\begin{equation}
  \label{eq_sg_CT}
  2x\,F^2 + (2x-1)F + 1 = 0.
\end{equation}

In the system of equations \eqref{systeme_gr4}, binary trees are
counted according to the number of internal nodes, whereas in the
context of operads, it is more natural to count trees by the number of
leaves. To compare Equation \eqref{eq_sg_CT} with Equation
\eqref{eq_sg_TS}, it is therefore necessary to replace $F$ by $F/x$ in
the former.

We shall show in section \ref{all_three_are_isomorphic} that this
generating series is also the generating series of $\opFF_4$. We first
need a third operad isomorphic to these first two, based on trees.

\begin{note}
  The presentation given here is quadratic and confluent. By choosing
  an appropriate monomial order, this gives a quadratic Groebner
  basis. By a known result~\cite{Dotsenko,Hoffbeck}, this implies that
  $\opGR$ is Koszul.
\end{note}

\section{The operads of red and white trees}
\label{sec_op_red_white_trees}
\subsection{Red and white trees}

Let us consider topological rooted trees, that is, rooted trees with
no order on the children, so that each node can have either none, one,
or multiple dots in it. If a node has no dots, then it must have at
least two children. The weight of a tree is its number of dots. This
set of trees is known in~\cite{Sloane} as Sequence~A050381, except for
the first term that is not $2$ but $1$.

The first values are
\begin{equation}
  1, 3, 10, 40, 170, 785, 3770, 18805,  96180, 502381, 2667034, 14351775, \dots 
\end{equation}

We shall make use of a small variation on these trees: the nodes will
be colored either red or white following the simple rule: a nonempty
node is always white and an empty node is red if and only if all its
children are white. Since there is only one such way to color the
nodes, this set is obviously in bijection with the previous one and is
called the \emph{red and white trees} and denoted by $\BWU$.
The set $\BWU(n)$ is the set of $\BWU$ trees of weight $n$.
Depending on what we are discussing, we shall make use of the red and white
trees or of the non-colored version.

We shall represent all trees in the following way: the red nodes are
represented in red and the others are represented in light blue
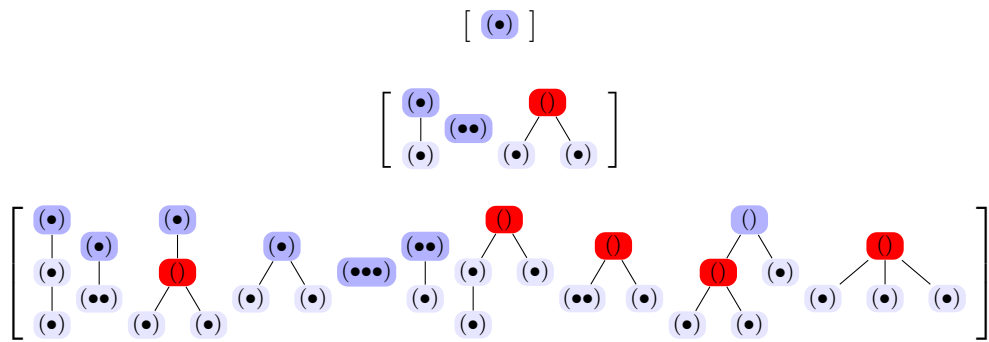
\footnote{When viewed in black-and-white, red appears dark and blue
  appears light.}.


\[\left[\ \begin{tikzpicture}
\node [root] {$\scriptstyle (\bullet)$};
\end{tikzpicture}\ \right]\]

\[\left[\ \begin{tikzpicture}
\node [root] {$\scriptstyle (\bullet)$}
  child {
    node [white] {$\scriptstyle (\bullet)$}
  };
\end{tikzpicture}\, \begin{tikzpicture}
\node [root] {$\scriptstyle (\bullet\bullet)$};
\end{tikzpicture}\, \begin{tikzpicture}
\node [root]  [red] {$\scriptstyle ()$}
  child {
    node [white] {$\scriptstyle (\bullet)$}
  }
  child {
    node [white] {$\scriptstyle (\bullet)$}
  };
\end{tikzpicture}\ \right]\]


\[\left[\ \begin{tikzpicture}
\node [root] {$\scriptstyle (\bullet)$}
  child {
    node [white] {$\scriptstyle (\bullet)$}
      child {
        node [white] {$\scriptstyle (\bullet)$}
      }
  };
\end{tikzpicture}\, \begin{tikzpicture}
\node [root] {$\scriptstyle (\bullet)$}
  child {
    node [white] {$\scriptstyle (\bullet\bullet)$}
  };
\end{tikzpicture}\, \begin{tikzpicture}
\node [root] {$\scriptstyle (\bullet)$}
  child {
    node  [red] {$\scriptstyle ()$}
      child {
        node [white] {$\scriptstyle (\bullet)$}
      }
      child {
        node [white] {$\scriptstyle (\bullet)$}
      }
  };
\end{tikzpicture}\, \begin{tikzpicture}
\node [root] {$\scriptstyle (\bullet)$}
  child {
    node [white] {$\scriptstyle (\bullet)$}
  }
  child {
    node [white] {$\scriptstyle (\bullet)$}
  };
\end{tikzpicture}\, \begin{tikzpicture}
\node [root] {$\scriptstyle (\bullet\bullet\bullet)$};
\end{tikzpicture}\, \begin{tikzpicture}
\node [root] {$\scriptstyle (\bullet\bullet)$}
  child {
    node [white] {$\scriptstyle (\bullet)$}
  };
\end{tikzpicture}\, \begin{tikzpicture}
\node [root]  [red] {$\scriptstyle ()$}
  child {
    node [white] {$\scriptstyle (\bullet)$}
      child {
        node [white] {$\scriptstyle (\bullet)$}
      }
  }
  child {
    node [white] {$\scriptstyle (\bullet)$}
  };
\end{tikzpicture}\, \begin{tikzpicture}
\node [root]  [red] {$\scriptstyle ()$}
  child {
    node [white] {$\scriptstyle (\bullet\bullet)$}
  }
  child {
    node [white] {$\scriptstyle (\bullet)$}
  };
\end{tikzpicture}\, \begin{tikzpicture}
\node [root] {$\scriptstyle ()$}
  child {
    node  [red] {$\scriptstyle ()$}
      child {
        node [white] {$\scriptstyle (\bullet)$}
      }
      child {
        node [white] {$\scriptstyle (\bullet)$}
      }
  }
  child {
    node [white] {$\scriptstyle (\bullet)$}
  };
\end{tikzpicture}\, \begin{tikzpicture}
\node [root]  [red] {$\scriptstyle ()$}
  child {
    node [white] {$\scriptstyle (\bullet)$}
  }
  child {
    node [white] {$\scriptstyle (\bullet)$}
  }
  child {
    node [white] {$\scriptstyle (\bullet)$}
  };
\end{tikzpicture}\ \right]\]

\subsubsection{Labelling red and white trees}

We shall now replace the dots by numbers and label our trees. The
first labelling is easy: if a tree belongs to $\BWU(n)$, replace each
dot by a different integer from $[n]=\{1,\dots,n\}$. We shall denote
this set of labellings by $\BWTSL(n)$. The number of such labellings
is Sequence~A005172 of~\cite{Sloane} by definition of this sequence,
and their first values are
\begin{equation}
  1, 4, 32, 416, 7552, 176128, 5018624,  168968192, 6563282944, 288909131776, \dots 
\end{equation}
One can give an equivalent definition using the language of
species~\cite{Specie} or labeled combinatorial classes~\cite{AnComb}. Let
$\SetSp$ be the species of sets (such that $\SetSp[U] := \{U\}$) and
$\SetSp_{\geq k}$ the species of sets of cardinality at least $k$. Let $\ZSp$
denote the singleton species (one object in size $1$). We denote by $+$ the
sum of species (disjoint union of labelled classes) and by $\cdot$ the product
of species (Cartesian product of labelled classes). Finally the substitution
is denoted functionally (as in $\mathcal{A}(\mathcal{B})$).

Then $\FSp := \BWTSL(n)$ is the solution of the equation
\begin{equation}
  \label{eq_FSp}
  \FSp = \SetSp_{\geq2}(\FSp) + \SetSp_{\geq1}(\ZSp)\cdot\SetSp(\FSp).
\end{equation}
As a consequence, their exponential generating function $F=F_{\BWTSL}$
satisfies
\begin{equation}
  \label{eq_sg_TSL}
  F = \exp(F) - 1 - F + (\exp(x)-1) \exp(F).
\end{equation}

\subsection{Operad structure on labelled red and white
  trees}
\label{sub-op-rw}

Let us consider two trees $T_1$ and $T_2$ and a label $x$ inside $T_1$
belonging to node $z$. To ensure that the composition of two trees is
a tree labelled by distinct integers, we shall first renumber $T_1$
and $T_2$ as follows: shift all labels greater than $x$ in $T_1$ by
the size of $T_2$, and shift all labels of $T_2$ by $x-1$.

The composition $T_1 \circ_x T_2$ is then defined as
\begin{itemize}
\item[(W)] If the root of $T_2$ is not red,
  erase $x$ in $z$, add the labels (if any) of the root of $T_2$ to $z$ and
  put the children of the root of $T_2$ as new children of $z$.
\item[(R)] If the root of $T_2$ is red, consider three cases:
  \begin{enumerate}
  \item[(R1)] If $T_1$ is the tree with one node labelled $x$, then the
    result is $T_2$.
  \item[(R2)] If $z$ is not a leaf or if $z$ is a leaf with multiple labels,
    start with $T_1$, remove $x$ from the labels of $z$ and glue the root of
    $T_2$ as a child of $z$.
  \item[(R3)] Otherwise, $z$ is a leaf and not the root, and $x$ is its only
              label.
    Consider $z'$, the parent of $z$. Then remove the leaf $z$ and
    put all children of the root of $T_2$ as new children of $z'$.
  \end{enumerate}
\end{itemize}
In all cases, if $z$ has no remaining labels, it is colored as white.

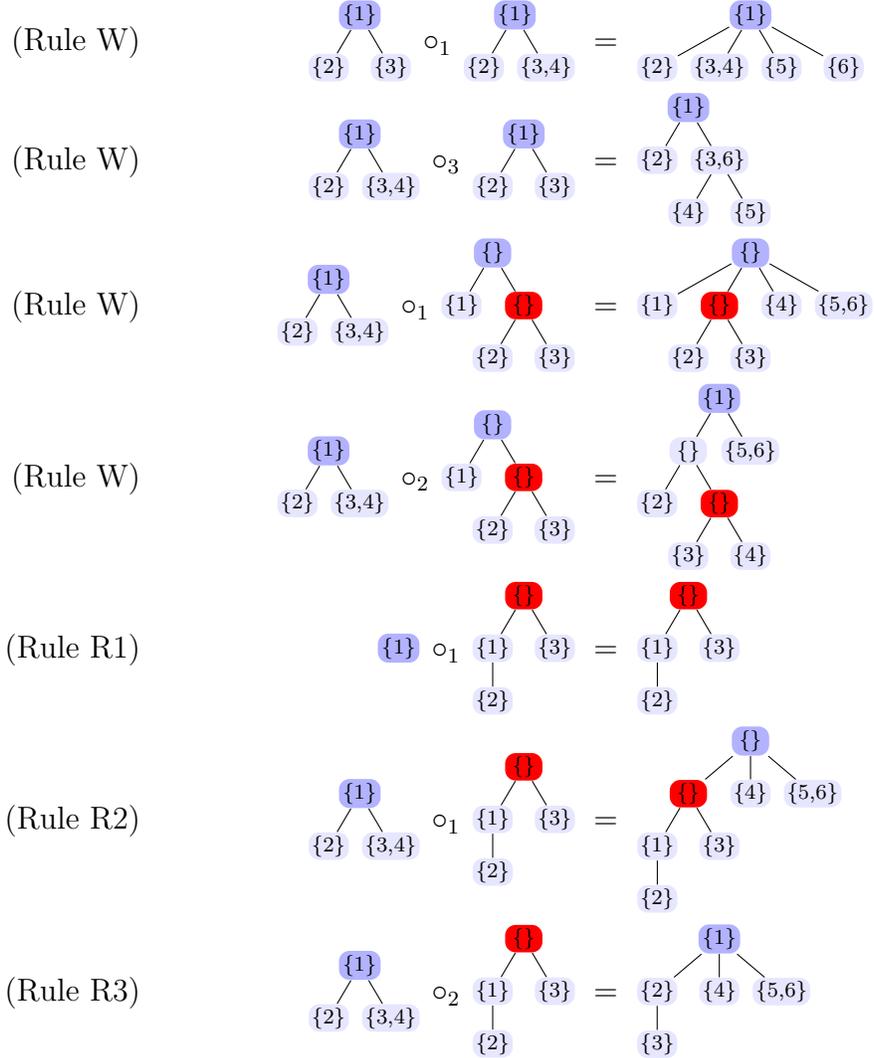
\def\elTree{\begin{tikzpicture}
\node [root]{$\scriptstyle \left\{1\right\}$}
  child {
    node [white] {$\scriptstyle \left\{2\right\}$}
  } 
  child {
    node [white] {$\scriptstyle \left\{3\right\}$}
  };
\end{tikzpicture}}
\def\elTreeQ{\begin{tikzpicture}
\node [root]{$\scriptstyle \left\{1\right\}$}
  child {
    node [white] {$\scriptstyle \left\{2\right\}$}
  }
  child {
    node [white] {$\scriptstyle \left\{3,4\right\}$}
  };
\end{tikzpicture}}
\def\whiteTree{\begin{tikzpicture}
\node [root]{$\scriptstyle \left\{\right\}$}
  child {
    node [white] {$\scriptstyle \left\{1\right\}$}
  }
  child {
    node  [red] {$\scriptstyle \left\{\right\}$}
      child {
        node [white] {$\scriptstyle \left\{2\right\}$}
      }
      child {
        node [white] {$\scriptstyle \left\{3\right\}$}
      }
  };
\end{tikzpicture}}
\def\blackTree{\begin{tikzpicture}
\node [root] [red] {$\scriptstyle \left\{\right\}$}
  child {
    node [white] {$\scriptstyle \left\{1\right\}$}
  }
  child {
    node [white] {$\scriptstyle \left\{3\right\}$}
      child {
        node [white] {$\scriptstyle \left\{2\right\}$}
      }
  };
\end{tikzpicture}}
\def\blackTreebis{\begin{tikzpicture}
\node [root] [red] {$\scriptstyle \left\{\right\}$}
  child {
    node [white] {$\scriptstyle \left\{1\right\}$}
      child {
        node [white] {$\scriptstyle \left\{2\right\}$}
      }
  }
  child {
    node [white] {$\scriptstyle \left\{3\right\}$}
  };
\end{tikzpicture}}
\begin{align*}
\text{(Rule W)}&&
\elTree\,\circ_1\,\elTreeQ
&\ =\ 
\begin{tikzpicture}
\node [root]{$\scriptstyle \left\{1\right\}$}
  child {
    node [white] {$\scriptstyle \left\{2\right\}$}
  }
  child {
    node [white] {$\scriptstyle \left\{3,4\right\}$}
  }
  child {
    node [white] {$\scriptstyle \left\{5\right\}$}
  }
  child {
    node [white] {$\scriptstyle \left\{6\right\}$}
  };
\end{tikzpicture}
\\
\text{(Rule W)}&&
\elTreeQ\,\circ_3\,\elTree
&\ =\ 
\begin{tikzpicture}
\node [root]{$\scriptstyle \left\{1\right\}$}
  child {
    node [white] {$\scriptstyle \left\{2\right\}$}
  }
  child {
    node [white] {$\scriptstyle \left\{3,6\right\}$}
      child {
        node [white] {$\scriptstyle \left\{4\right\}$}
      }
      child {
        node [white] {$\scriptstyle \left\{5\right\}$}
      }
  };
\end{tikzpicture}
\\
\text{(Rule W)}&&
\elTreeQ\,\circ_1\,\whiteTree
&\ =\ 
\begin{tikzpicture}
\node [root]{$\scriptstyle \left\{\right\}$}
  child {
    node [white] {$\scriptstyle \left\{1\right\}$}
  }
  child {
    node  [red] {$\scriptstyle \left\{\right\}$}
      child {
        node [white] {$\scriptstyle \left\{2\right\}$}
      }
      child {
        node [white] {$\scriptstyle \left\{3\right\}$}
      }
  }
  child {
    node [white] {$\scriptstyle \left\{4\right\}$}
  }
  child {
    node [white] {$\scriptstyle \left\{5,6\right\}$}
  };
\end{tikzpicture}
\\
\text{(Rule W)}&&
\elTreeQ\,\circ_2\,\whiteTree
&\ =\ 
\begin{tikzpicture}
\node [root]{$\scriptstyle \left\{1\right\}$}
  child {
    node [white] {$\scriptstyle \left\{\right\}$}
      child {
        node [white] {$\scriptstyle \left\{2\right\}$}
      }
      child {
        node  [red] {$\scriptstyle \left\{\right\}$}
          child {
            node [white] {$\scriptstyle \left\{3\right\}$}
          }
          child {
            node [white] {$\scriptstyle \left\{4\right\}$}
          }
      }
  }
  child {
    node [white] {$\scriptstyle \left\{5,6\right\}$}
  }
  ;
\end{tikzpicture}
\\
%
\text{(Rule R1)}&&
\begin{tikzpicture}
\node [root]{$\scriptstyle \left\{1\right\}$};
\end{tikzpicture}
\,\circ_1\,\blackTreebis
&\ =\ 
\blackTreebis
\\
\text{(Rule R2)}&&
\elTreeQ\,\circ_1\,\blackTreebis
&\ =\ 
\begin{tikzpicture}
\node [root]{$\scriptstyle \left\{\right\}$}
  child {
    node  [red] {$\scriptstyle \left\{\right\}$}
      child {
        node [white] {$\scriptstyle \left\{1\right\}$}
          child {
            node [white] {$\scriptstyle \left\{2\right\}$}
          }
      }
      child {
        node [white] {$\scriptstyle \left\{3\right\}$}
      }
  }
  child {
    node [white] {$\scriptstyle \left\{4\right\}$}
  }
  child {
    node [white] {$\scriptstyle \left\{5,6\right\}$}
  };
\end{tikzpicture}
\\
\text{(Rule R3)}&&
\elTreeQ\,\circ_2\,\blackTreebis
&\ =\ 
\begin{tikzpicture}
\node [root]{$\scriptstyle \left\{1\right\}$}
  child {
    node [white] {$\scriptstyle \left\{2\right\}$}
      child {
        node [white] {$\scriptstyle \left\{3\right\}$}
      }
  }
  child {
    node [white] {$\scriptstyle \left\{4\right\}$}
  }
  child {
    node [white] {$\scriptstyle \left\{5,6\right\}$}
  };
\end{tikzpicture}
\end{align*}

\begin{theorem}
  The set $\BWTSL$ endowed with operations $\circ_x$ is a symmetric
  set-operad.
\end{theorem}

\Proof
Let us first check that that the $\circ_i$ are internal. The only problem that
could arise would be to have two red nodes as father and child, one being
created by emptying a node of the first tree. But any new empty node has at
least two children, so that this cannot happen.

Let us now check that the operations $\circ_x$ satisfy the axioms of an
operad, that is:
if $x$ and $y$ are two labels of a tree $A$,
and $t$ a label of a tree $B$,
\begin{equation}
(A\circ_x B)\circ_y C =
(A\circ_y C)\circ_x B,
\quad
(A\circ_x B)\circ_t C =
A\circ_x (B\circ_t C).
\end{equation}
Note that when applying the composition rule $A\circ_x B$, only the node
labelled $x$ in $A$ in changed. So the first equation is trivial except in the
case where $x$ and $y$ are in the same node $z$. In that case, we need to
consider all four cases for the roots of $B$ and $C$. If both are not red,
case $(W)$ applies twice and the result is obvious. If one is red, say the
root of $B$ and not the other, then case $(R2)$ applies twice when
substituting $B$ since $z$ cannot both be a leaf and have only one label.
Case $W$ applies twice when substituting $C$ so the result of both hand-sides
is the same.
Now, if both roots of $B$ and $C$ are red, for the same reason as before,
both substituting go by application of rule $(R2)$ again, so the same result
holds.

\medskip
Let us now check the second rule of operads.
If any tree is a single root with one label, the equality is clear.
Now, the root of $B$ can never become red after $B\circ_z C$.
So the substitution of $B$ into $A$ goes along the same case as the
substitution of $B\circ_z C$ into $A$ since the distinction between the $(R)$
cases only relies on properties of $A$. Now, the substitution of $C$ into $B$
also in obtained by applying the same case as its substitution into
$A\circ_x B$, so that the result on both sides is always the same.
\qed

\begin{note}
  In fact, this composition defines a symmetric operad, which will be
  considered in Section~\ref{symmop}.
\end{note}

\subsection{The operad $\BWTS$ and recursively labelled red-white trees}

We shall be interested now in the sub-operad $\BWTS$ of the operad
$\BWTSL$ on all labelled red and white trees, which is generated by
the four elements of size $2$ (and containing the element of size
$1$):

\begin{equation}
\left[ \geng12, \geno12, \gend12, \gens12
\right].
\end{equation}

\subsubsection{Recursively labelled red-white trees}

\begin{definition}
  We say that a labelling of a tree in $\BWU(n)$ is \emph{recursive}
  if, for any node, the set of labels of all its descendants
  (including its own labels) is an interval of $[n]$.
\end{definition}
We shall denote this set of labellings by $\BWTS(n)$.  The number of
such labeled trees is given by Sequence~A156017 of~\cite{Sloane} and
their first values are
\begin{equation}
  1, 4, 24, 176, 1440, 12608, 115584, 1095424, 10646016, 105522176, 1062623232, \dots 
\end{equation}
Here are the first examples: 
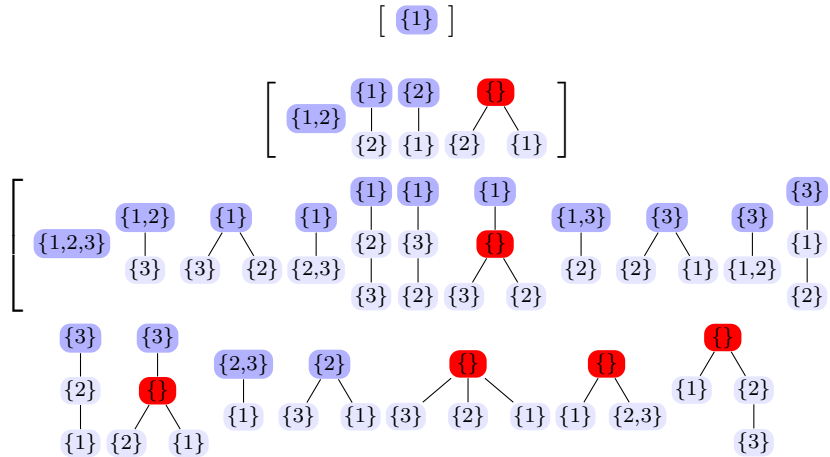
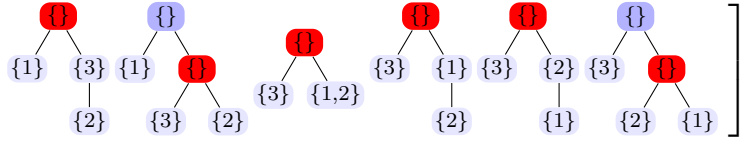
\[\left[\ \begin{tikzpicture}
    \node [root]{$\scriptstyle \left\{1\right\}$};
\end{tikzpicture}\ \right]\]

\[\left[\ \begin{tikzpicture}
\node [root]{$\scriptstyle \left\{1, 2\right\}$};
\end{tikzpicture}\, \begin{tikzpicture}
\node [root]{$\scriptstyle \left\{1\right\}$}
  child {
    node [white] {$\scriptstyle \left\{2\right\}$}
  };
\end{tikzpicture}\, \begin{tikzpicture}
\node [root]{$\scriptstyle \left\{2\right\}$}
  child {
    node [white] {$\scriptstyle \left\{1\right\}$}
  };
\end{tikzpicture}\, \begin{tikzpicture}
\node [fill=blue!30] [red] {$\scriptstyle \left\{\right\}$}
  child {
    node [white] {$\scriptstyle \left\{2\right\}$}
  }
  child {
    node [white] {$\scriptstyle \left\{1\right\}$}
  };
\end{tikzpicture}\ \right]\]

\[\left[\ \begin{tikzpicture}
\node [root]{$\scriptstyle \left\{1, 2, 3\right\}$};
\end{tikzpicture}\, \begin{tikzpicture}
\node [root]{$\scriptstyle \left\{1, 2\right\}$}
  child {
    node [white] {$\scriptstyle \left\{3\right\}$}
  };
\end{tikzpicture}\, \begin{tikzpicture}
\node [root]{$\scriptstyle \left\{1\right\}$}
  child {
    node [white] {$\scriptstyle \left\{3\right\}$}
  }
  child {
    node [white] {$\scriptstyle \left\{2\right\}$}
  };
\end{tikzpicture}\, \begin{tikzpicture}
\node [root]{$\scriptstyle \left\{1\right\}$}
  child {
    node [white] {$\scriptstyle \left\{2, 3\right\}$}
  };
\end{tikzpicture}\, \begin{tikzpicture}
\node [root]{$\scriptstyle \left\{1\right\}$}
  child {
    node [white] {$\scriptstyle \left\{2\right\}$}
      child {
        node [white] {$\scriptstyle \left\{3\right\}$}
      }
  };
\end{tikzpicture}\, \begin{tikzpicture}
\node [root]{$\scriptstyle \left\{1\right\}$}
  child {
    node [white] {$\scriptstyle \left\{3\right\}$}
      child {
        node [white] {$\scriptstyle \left\{2\right\}$}
      }
  };
\end{tikzpicture}\, \begin{tikzpicture}
\node [root]{$\scriptstyle \left\{1\right\}$}
  child {
    node  [red] {$\scriptstyle \left\{\right\}$}
      child {
        node [white] {$\scriptstyle \left\{3\right\}$}
      }
      child {
        node [white] {$\scriptstyle \left\{2\right\}$}
      }
  };
\end{tikzpicture}\, \begin{tikzpicture}
\node [root]{$\scriptstyle \left\{1, 3\right\}$}
  child {
    node [white] {$\scriptstyle \left\{2\right\}$}
  };
\end{tikzpicture}\, \begin{tikzpicture}
\node [root]{$\scriptstyle \left\{3\right\}$}
  child {
    node [white] {$\scriptstyle \left\{2\right\}$}
  }
  child {
    node [white] {$\scriptstyle \left\{1\right\}$}
  };
\end{tikzpicture}\, \begin{tikzpicture}
\node [root]{$\scriptstyle \left\{3\right\}$}
  child {
    node [white] {$\scriptstyle \left\{1, 2\right\}$}
  };
\end{tikzpicture}\, \begin{tikzpicture}
\node [root]{$\scriptstyle \left\{3\right\}$}
  child {
    node [white] {$\scriptstyle \left\{1\right\}$}
      child {
        node [white] {$\scriptstyle \left\{2\right\}$}
      }
  };
\end{tikzpicture}
\right.\]
\[\left.
\begin{tikzpicture}
\node [root]{$\scriptstyle \left\{3\right\}$}
  child {
    node [white] {$\scriptstyle \left\{2\right\}$}
      child {
        node [white] {$\scriptstyle \left\{1\right\}$}
      }
  };
\end{tikzpicture}\, \begin{tikzpicture}
\node [root]{$\scriptstyle \left\{3\right\}$}
  child {
    node  [red] {$\scriptstyle \left\{\right\}$}
      child {
        node [white] {$\scriptstyle \left\{2\right\}$}
      }
      child {
        node [white] {$\scriptstyle \left\{1\right\}$}
      }
  };
\end{tikzpicture}\, \begin{tikzpicture}
\node [root]{$\scriptstyle \left\{2, 3\right\}$}
  child {
    node [white] {$\scriptstyle \left\{1\right\}$}
  };
\end{tikzpicture}\, \begin{tikzpicture}
\node [root]{$\scriptstyle \left\{2\right\}$}
  child {
    node [white] {$\scriptstyle \left\{3\right\}$}
  }
  child {
    node [white] {$\scriptstyle \left\{1\right\}$}
  };
\end{tikzpicture}\, \begin{tikzpicture}
\node [fill=blue!30] [red] {$\scriptstyle \left\{\right\}$}
  child {
    node [white] {$\scriptstyle \left\{3\right\}$}
  }
  child {
    node [white] {$\scriptstyle \left\{2\right\}$}
  }
  child {
    node [white] {$\scriptstyle \left\{1\right\}$}
  };
\end{tikzpicture}\, \begin{tikzpicture}
\node [fill=blue!30] [red] {$\scriptstyle \left\{\right\}$}
  child {
    node [white] {$\scriptstyle \left\{1\right\}$}
  }
  child {
    node [white] {$\scriptstyle \left\{2, 3\right\}$}
  };
\end{tikzpicture}\, \begin{tikzpicture}
\node [fill=blue!30] [red] {$\scriptstyle \left\{\right\}$}
  child {
    node [white] {$\scriptstyle \left\{1\right\}$}
  }
  child {
    node [white] {$\scriptstyle \left\{2\right\}$}
      child {
        node [white] {$\scriptstyle \left\{3\right\}$}
      }
  };
\end{tikzpicture}
\right.\]
\[\left.
\begin{tikzpicture}
\node [fill=blue!30] [red] {$\scriptstyle \left\{\right\}$}
  child {
    node [white] {$\scriptstyle \left\{1\right\}$}
  }
  child {
    node [white] {$\scriptstyle \left\{3\right\}$}
      child {
        node [white] {$\scriptstyle \left\{2\right\}$}
      }
  };
\end{tikzpicture}\, \begin{tikzpicture}
\node [root]{$\scriptstyle \left\{\right\}$}
  child {
    node [white] {$\scriptstyle \left\{1\right\}$}
  }
  child {
    node  [red] {$\scriptstyle \left\{\right\}$}
      child {
        node [white] {$\scriptstyle \left\{3\right\}$}
      }
      child {
        node [white] {$\scriptstyle \left\{2\right\}$}
      }
  };
\end{tikzpicture}\, \begin{tikzpicture}
\node [fill=blue!30] [red] {$\scriptstyle \left\{\right\}$}
  child {
    node [white] {$\scriptstyle \left\{3\right\}$}
  }
  child {
    node [white] {$\scriptstyle \left\{1, 2\right\}$}
  };
\end{tikzpicture}\, \begin{tikzpicture}
\node [fill=blue!30] [red] {$\scriptstyle \left\{\right\}$}
  child {
    node [white] {$\scriptstyle \left\{3\right\}$}
  }
  child {
    node [white] {$\scriptstyle \left\{1\right\}$}
      child {
        node [white] {$\scriptstyle \left\{2\right\}$}
      }
  };
\end{tikzpicture}\, \begin{tikzpicture}
\node [fill=blue!30] [red] {$\scriptstyle \left\{\right\}$}
  child {
    node [white] {$\scriptstyle \left\{3\right\}$}
  }
  child {
    node [white] {$\scriptstyle \left\{2\right\}$}
      child {
        node [white] {$\scriptstyle \left\{1\right\}$}
      }
  };
\end{tikzpicture}\, \begin{tikzpicture}
\node [root]{$\scriptstyle \left\{\right\}$}
  child {
    node [white] {$\scriptstyle \left\{3\right\}$}
  }
  child {
    node  [red] {$\scriptstyle \left\{\right\}$}
      child {
        node [white] {$\scriptstyle \left\{2\right\}$}
      }
      child {
        node [white] {$\scriptstyle \left\{1\right\}$}
      }
  };
\end{tikzpicture}\ \right]\]
Let $F:=F_{\BWTS}$ be the ordinary generating series of such trees. Because of
the recursive labeling, trees having an empty root (and thus at least two
sub-trees) are in bijection with sequences of length at least 2 of trees as
follows: the sequence $(t_1, t_2, \dots, t_r)$ with $r\geq2$ of trees of sizes
$(s_1, \dots, s_r)$, corresponds to the following trees
\begin{equation}
  (t_1, t_2, \dots, t_r)
  \quad\longleftrightarrow\quad
  \begin{tikzpicture}[level distance = 1.2cm, sibling distance =1.5cm]
    \node  {
      \begin{tikzpicture}\node [red] {\{\}};\end{tikzpicture} 
      or
      \begin{tikzpicture}\node [white] {\{\}};\end{tikzpicture}}
    child { node {$t_1$} }
    child { node {$t_2[s_1]$} }
    child { node {\dots} }
    child { node {\hbox to 1cm{$t_r[s_1+\dots+s_{r-1}]\hss$}} };
  \end{tikzpicture}
\end{equation}
where $t[k]$ denote the tree obtained from $t$ by adding $k$ to all
integers in the labels. Therefore the generating series of those trees
is given by $1/(1-F)-1-F$. In a similar fashion, there is a bijection
between trees and non-empty sequences of either a dot (labelled $1$) or
a tree, by shifting the labels (including dots) as previously and
gathering dots into the roots. Here is an example:
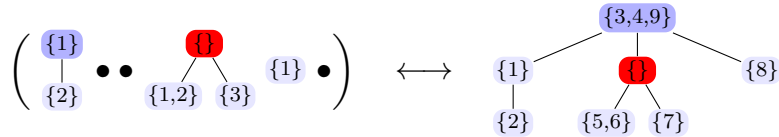
\[
\left(\
\begin{tikzpicture}
\node [root]{$\scriptstyle \left\{1\right\}$}
  child {
    node [white] {$\scriptstyle \left\{2\right\}$}
  };
\end{tikzpicture}\,
\bullet
\bullet\
\begin{tikzpicture}
\node [fill=blue!30] [red] {$\scriptstyle \left\{\right\}$}
  child {
    node [white] {$\scriptstyle \left\{1, 2\right\}$}
  }
  child {
    node [white] {$\scriptstyle \left\{3\right\}$}
  };
\end{tikzpicture}\
\begin{tikzpicture}
\node [fill=blue!30] [white] {$\scriptstyle \left\{1\right\}$};
\end{tikzpicture}\
\bullet
\right)
\quad\longleftrightarrow\quad
\begin{tikzpicture}
  \node [root]{$\scriptstyle \left\{3,4,9\right\}$}
  child {
    node [white]{$\scriptstyle \left\{1\right\}$}
    child {
      node [white] {$\scriptstyle \left\{2\right\}$}
    }
  }
  child [missing]
  child {
    node [red] {$\scriptstyle \left\{\right\}$}
    child {
      node [white] {$\scriptstyle \left\{5, 6\right\}$}
    }
    child {
      node [white] {$\scriptstyle \left\{7\right\}$}
    }
  }
  child [missing]
  child {
    node [white] {$\scriptstyle \left\{8\right\}$}
  };
\end{tikzpicture}
\]
Since we are here only considering trees without empty root, their
generating series is therefore $(x+F)/(1-(x+F))-F/(1-F)$. Summarizing
the previous reasoning we get that, the ordinary generating series
$F:=F_{\BWTS}$ satisfies
\begin{equation}
  \label{eq_sg_TS_first}
  F= 1/(1-F) - 1 - F + (x+F)/(1-(x+F)) - F/(1-F),
\end{equation}
that is, after simplifications and clearing the denominator
\begin{equation}
  \label{eq_sg_TS}
  F = x + 2x F + 2F^2.
\end{equation}

One then easily checks that this equation is the same as
Equation~\eqref{eq_sg_CT} (up to multiplication by $x$) obtained for
the number of canonical $\opGR$ trees, hence suggesting that a natural
operad isomorphism exists between the two structures.
We shall first check that the recursively labelled trees are indeed
elements of an operad.

\subsubsection{The sub-operad $\BWTS$}

\begin{proposition}
  \label{prop-op4-bw}
  The elements of $\BWTS$ are the recursively labelled red and white trees.
\end{proposition}

\Proof
First, let us note that the composition $\circ_x$ of two recursively labelled
red and white trees $T_1$ and $T_2$ gives rise to a recursively labelled red
and white tree $T$.
Indeed, $T$ is a labelled tree. Now, if the node that contains $x$ on $T_1$
has as set of labels of descendants $[y,z]$, the corresponding node in $T$ has
as set of labels of descendants $[y,z+|T_2|]$. Now, since all labels greater
than $x$ in $T_1$ have been shifted by the size $|T_2|$ of $T_2$, all nodes of
$T$ also have an interval as set of labels of their descendants.

Now, let us prove the converse, that is, that all recursively labelled red and
white trees can be obtained as substitutions of smaller trees.
The property is true for trees of size at most $2$.
Now, if a node $z$ of a tree $T$ of size $k$ has more than two labels, $T$ is
easily written as a substitution of smaller trees on this node: indeed, $T$
is obtained by substituting the tree with a root labelled $\{1,2\}$ in the
tree $T'$ obtained from $T$ by removing any of the consecutive integers
labelling $z$ and then renumbering all labels in order to obtain all numbers
from $1$ to $k-1$.

Otherwise, consider a leaf $l$ that is not an uncle/aunt.
Let us consider cases depending on the parent $p$ of $l$:
\begin{itemize}
\item If $p$ has $x$ as a label, then $T=T'\circ_y T''$ where $T'$ is obtained
  from $T$ by changing $x$ into $y$ and removing $l$, and $T''$ is the tree
  having as root $x$ with one child $l$.
\end{itemize}
Otherwise, $p$ is empty, and it is necessarily red because if $l$ has some
siblings then they are leaves and thus white.
\begin{itemize}
\item Now, if $l$ has two siblings or more, then $T=T'\circ_y
  T''$ where $T'$ is obtained from $T$ by labelling $l$ by $y$ and removing
  one sibling of $l$, and $T''$ is a tree with a red root and two children:
  $l$ and the missing sibling of $l$.
\item Otherwise, $l$ has only one sibling, $p$ has itself a parent $p'$, which
  is white, either labelled or not. In both cases, $T$ is obtained as
  $T=T'\circ_y T''$ where $T'$ is obtained from $T$ by changing $p$ and its
  sub-tree by a new leaf labelled $y$ and where $T''$ is the sub-tree of $T$ of
  root $p$.
\end{itemize}
We conclude by induction on the size of the trees.
\qed

\begin{theorem}
The operads $\BWTS$ and $\opGR$ are isomorphic as operads.
\end{theorem}

\Proof
We already know that as sets, $\BWTS$ and $\opGR$ are equinumerous.
Since $\opGR$ is the quotient of the free operad on four generators
with eight relations, we only need to prove that the generators of $\opGR$
satisfy the same relations.
The decoding is the following:
\begin{equation}
\left[\geno12,\geng12,\gend12,\gens12
\right]
\Longleftrightarrow
[\treeone{\circ},\treeone{\prec},\treeone{\succ},\treeone{\odot}].
\end{equation}
The fact that all relations hold is immediate given the definition of the
operad $\BWTS$.
For example, the first relation of Equation \eqref{eq.def.GR}
\begin{equation}
(x\droittrid y)\gautrid z = x\droittrid (y\gautrid z)\,,\\
\end{equation}
rewrites as
\begin{equation}
(a\gautrid z)\circ_a (x\droittrid y)
 = (x\droittrid a) \circ_a (y\gautrid z)\,,
\end{equation}
which rewrites in the red and white trees as
\begin{equation}
\geng az\circ_a \gend xy = \gend xa \circ_a \geng yz\,,
\end{equation}
which is true since both expressions are equal to
\begin{equation}
\begin{tikzpicture}
\node [root] {$\scriptstyle \left\{y\right\}$}
  child {
    node [white] {$\scriptstyle \left\{x\right\}$}
  }
  child {
    node [white] {$\scriptstyle \left\{z\right\}$}
  };
\end{tikzpicture}
\end{equation}
The trees corresponding to the eight relations are

\begin{equation}
\begin{tikzpicture}
\node [root] {$\scriptstyle \left\{y\right\}$}
  child {
    node [white] {$\scriptstyle \left\{x\right\}$}
  }
  child {
    node [white] {$\scriptstyle \left\{z\right\}$}
  };
\end{tikzpicture},
\
\begin{tikzpicture}
    \node [root]{$\scriptstyle \left\{xyz\right\}$};
\end{tikzpicture},
\
\begin{tikzpicture}
\node [root] {$\scriptstyle \left\{yz\right\}$}
  child {
    node [white] {$\scriptstyle \left\{x\right\}$}
  };
\end{tikzpicture},
\
\begin{tikzpicture}
\node [root] {$\scriptstyle \left\{xz\right\}$}
  child {
    node [white] {$\scriptstyle \left\{y\right\}$}
  };
\end{tikzpicture},
\
\begin{tikzpicture}
\node [root] {$\scriptstyle \left\{xy\right\}$}
  child {
    node [white] {$\scriptstyle \left\{z\right\}$}
  };
\end{tikzpicture},
\
\begin{tikzpicture}
\node [red] {$\scriptstyle \left\{\right\}$}
  child { node [white] {$\scriptstyle \left\{x\right\}$} } 
  child { node [white] {$\scriptstyle \left\{y\right\}$} } 
  child { node [white] {$\scriptstyle \left\{z\right\}$} };
\end{tikzpicture},
\
\begin{tikzpicture}
\node [root] {$\scriptstyle \left\{x\right\}$}
  child { node [white] {$\scriptstyle \left\{y\right\}$} } 
  child { node [white] {$\scriptstyle \left\{z\right\}$} };
\end{tikzpicture},
\
\begin{tikzpicture}
\node [root] {$\scriptstyle \left\{z\right\}$}
  child { node [white] {$\scriptstyle \left\{x\right\}$} } 
  child { node [white] {$\scriptstyle \left\{y\right\}$} };
\end{tikzpicture}.
\end{equation}
\qed

\subsection{From red and white trees to $\opFF_4$}
\label{sec_bijections}

Let us summarize what we have in terms of operads.
We first have the operad $\opGR$ that is sent surjectively to the operad
$\opFF_4$ since $\opGR$ is the quotient of the free operad on four generators
with eight relations and that $\opFF_4$ has four generators and satisfies the
eight relations. We also have that $\opGR$ and $\BWTS$ are isomorphic
operads on equinumerous sets.
To conclude that $\opFF_4$ is isomorphic to $\opGR$, it only remains
to prove that there exists an injective morphism of operads $\Phi$
from $\BWTS$ to $\opFF_4$ that makes the diagram below commutative.
\begin{center}
\begin{tikzpicture}
\label{diagramme_commut}
  \matrix (m) [matrix of math nodes,row sep=3em,column sep=4em,minimum width=2em]
  {
     \opGR & \opFF_4 \\
     \BWTS & \opFF \\};
  \path[-stealth]
    (m-1-1) edge node [left] {$\simeq$} (m-2-1)
            edge [->>] (m-1-2)
    (m-2-1) edge [>->] node [above] {$\Phi$} (m-1-2)
    (m-1-2) edge [>->] (m-2-2)
    (m-2-1) edge [>->] node [above] {$\Phi$} (m-2-2);
\end{tikzpicture}
\end{center}

Our injective morphism is a set morphism, that is, a bijection.

\subsubsection{An injection from $\BWTSL$ to the fractions in $\opFF$}
\label{sec_bij_RW_Frac}

We define here a map from labelled red and white trees to formal
fractions, that will be a bijection with its image.

The bijection is as follows. Consider a red and white tree $T$.
For each node $z$, compute the set $S(z)$ of all values in the sub-tree of root
$z$ and then define $E(z)$ as either
\begin{equation}
\left\{
\begin{array}{ll}
\frac{1}{[S(z)]} & \text{if $z$ is white}, \\[5pt]
[S(z)]  & \text{if $z$ is red and not the root},\\
1 & \text{if $z$ is red and the root}.
\end{array}
\right.
\end{equation}

The fraction associated with $T$ is then
\begin{equation}
\label{eq_bij_RW_Frac}
\Phi(T) := \prod_{z\in nodes(T)} E(z).
\end{equation}

For example, we have
\[\begin{tikzpicture}
\node [root] {$\scriptstyle \left\{5\right\}$}
  child {
    node [white] {$\scriptstyle \left\{1, 2\right\}$}
      child {
        node [red] {$\scriptstyle \left\{\right\}$}
          child {
            node [white] {$\scriptstyle \left\{3\right\}$}
          }
          child {
            node [white] {$\scriptstyle \left\{4\right\}$}
          }
      }
  }
  child [missing]
  child {
    node [white] {$\scriptstyle \left\{\right\}$}
      child {
        node [white] {$\scriptstyle \left\{8\right\}$}
      }
      child {
        node [red] {$\scriptstyle \left\{\right\}$}
          child {
            node [white] {$\scriptstyle \left\{6\right\}$}
          }
          child {
            node [white] {$\scriptstyle \left\{7\right\}$}
          }
      }
  };
\end{tikzpicture}\qquad
\frac{\mbox{[34]} \mbox{[67]}}{\mbox{[3]} \mbox{[4]} \mbox{[1234]}
\mbox{[6]} \mbox{[7]} \mbox{[8]} \mbox{[678]} \mbox{[12345678]}}\]
\medskip

\[\begin{tikzpicture}
\node [red] {$\scriptstyle \left\{\right\}$}
  child {
    node [white] {$\scriptstyle \left\{1\right\}$}
  }
  child [missing]
  child {
    node [white] {$\scriptstyle \left\{\right\}$}
      child {
        node [white] {$\scriptstyle \left\{5\right\}$}
          child {
            node [white] {$\scriptstyle \left\{4\right\}$}
              child {
                node [white] {$\scriptstyle \left\{2\right\}$}
              }
              child {
                node [white] {$\scriptstyle \left\{3\right\}$}
              }
          }
      }
      child [missing]
      child {
        node [red] {$\scriptstyle \left\{\right\}$}
          child {
            node [white] {$\scriptstyle \left\{6\right\}$}
          }
          child {
            node [white] {$\scriptstyle \left\{7, 8\right\}$}
          }
      }
  };
\end{tikzpicture}\qquad
\frac{\mbox{[678]}}{\mbox{[1]}  \mbox{[2]} \mbox{[3]}
 \mbox{[2345]} \mbox{[234]} \mbox{[6]} \mbox{[78]}  \mbox{[2345678]}}\]

\begin{lemma}
  \label{lemma_Phi_injective}
  $\Phi$ is an injective map from $\BWTSL$ to $\opFF$.
\end{lemma}
\begin{proof}
  We prove that one easily rebuilds the tree from the fraction: let us show
  how to rebuild the root, the rest of the construction being done
  recursively.  Let $F$ be a fraction obtained by the previous process.  Let
  $S$ be the union of all values in the fraction.

  Either $[S]$ belongs to the fraction or not. If not, the root is
  red, separate $[S]$ in the greatest possible number of sets (into
  the coarsest partition) so that any element of $F$ is a subset of
  one of these sets. These are the children of the root. Iterate the
  process on each child separately.

  If $[S]$ belongs to $F$, then it is on the denominator of $F$.  Then the
  root is white and labelled by $[S]$ minus the union of all values of the
  fraction $F'=[S] F$. Note that the root might be empty.  Then split $F'$ as
  in the case of the red root and iterate.
\end{proof}
\subsubsection{The main theorem}
\label{all_three_are_isomorphic}

\begin{theorem}
\label{theorem_main}
The three set-operads $\BWTS$ on red and white trees,
$\opFF_4$ on formal fractions and $\opGR$ are isomorphic.
\end{theorem}

\Proof
We first show that the bijection $\Phi$ between $\BWTS$ and some
formal fractions is compatible with the operad $\circ_i$ operations,
meaning that it is a morphism of operads from $\BWTS$ to $\opFF$.

Let us check this for the various cases $(W)$, $(R1)$, $(R2)$ and
$(R3)$ in the composition $T_1 \circ_x T_2$ of elements of
$\BWTS$. Let $z$ be the vertex of $T_1$ containing $x$, which is
necessarily non-empty and therefore white. Let $f_1$ and $f_2$ be the
fractions associated to $T_1$ and $T_2$ by $\Phi$. Let $\hat{f}_1$ be
the fraction obtained from $f_1$ by replacing $x$ by the indices of
$T_2$. Let $S_2$ be the set of indices of $T_2$.

In the $(W)$ case, the root of $T_2$ is not red and the vertex
corresponding to $z$ in $T_1 \circ_x T_2$ remains white. By the
description of $(R2)$ and the definition of $\Phi$, the fraction
associated to $T_1 \circ_x T_2$ is the product of $f_2 [S_2]$ (coming
from vertices of $T_2$ except the white root) and $\hat{f_1}$ (coming
from vertices of $T_1$).

The $(R1)$ case follows from the fact that the unit of the operad
$\BWTS$ is mapped to the unit of the operad $\opFF$.

In the $(R2)$ case, the root of $T_2$ is red, and the vertex
corresponding to $z$ in $T_1 \circ_x T_2$ remains white. By the
description of $(R2)$ and the definition of $\Phi$, the fraction
associated to $T_1 \circ_x T_2$ is the product of $\hat{f}_1$ (from
vertices coming from $T_1$), $f_2$ (from vertices coming from $T_2$
other than the root of $T_2$) and $[S_2]$ (coming from the red root of
$T_2$, which becomes a non-root red vertex).

In the $(R3)$ case, the root of $T_2$ is red. By the description of
$(R2)$ and the definition of $\Phi$, the fraction associated to $T_1
\circ_x T_2$ is the product of $f_2$ (from vertices coming from $T_2$
other than the root of $T_2$) and $\hat{f}_1 [S_2]$ (from vertices
coming from $T_1$, except the vertex $z$ which was white and is removed)

In all cases, the fraction obtained is the same as the fraction $f_1
\circ_x f_2$ which is $[S_2] \hat{f}_1 f_2$ by the composition rule
\eqref{def.composition.ff} of $\opFF$. This proves that $\Phi$ is a
morphism of operads from $\BWTS$ to $\opFF$.

One can check on the four generators of $\opGR$ that the inclusion of
$\opGR$ in $\opFF_4$ is the same as the composite of $\Phi$ and the
isomorphism of $\opGR$ and $\BWTS$.

This also proves that the image of $\BWTS$ by the bijection $\Phi$
is the same as the image of the set-operad $\opGR$ since all fractions
can be obtained by applying the substitutions $\circ_i$ to the
generators.
\qed



Here are a $\opGR$ tree, its corresponding red and white tree and its
fraction.

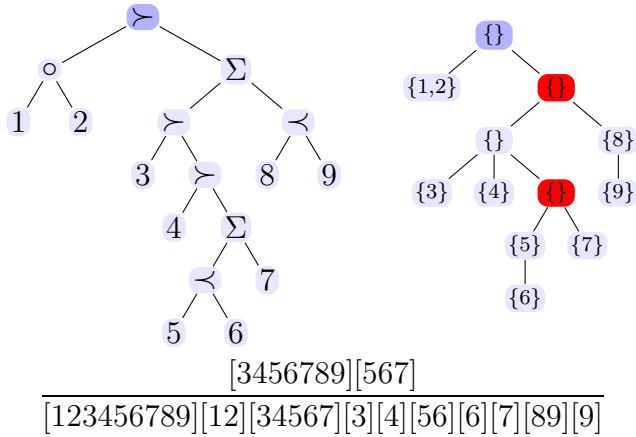
\begin{gather*}
\begin{tikzpicture}[baseline=(current bounding box.east),
  every node/.style={inner sep=1.5pt, fill=blue!10, rounded corners}]
\node [fill=blue!30]{$\succ$}
  child {
    node {$\circ$}
      child {
        node {$1$}
      }
      child {
        node {$2$}
      }
  }
  child [missing]
  child [missing]
  child {
    node {$\Sigma$}
      child {
        node {$\succ$}
          child {
            node {$3$}
          }
          child {
            node {$\succ$}
              child {
                node {$4$}
              }
              child {
                node {$\Sigma$}
                  child {
                    node {$\prec$}
                      child {
                        node {$5$}
                      }
                      child {
                        node {$6$}
                      }
                  }
                  child {
                    node {$7$}
                  }
              }
          }
      }
      child [missing]
      child {
        node {$\prec$}
          child {
            node {$8$}
          }
          child {
            node {$9$}
          }
      }
  };
\end{tikzpicture}
\qquad
\begin{tikzpicture}
\node [root] {$\scriptstyle \left\{\right\}$}
  child {
    node [white] {$\scriptstyle \left\{1, 2\right\}$}
  }
  child [missing]
  child {
    node [red] {$\scriptstyle \left\{\right\}$}
      child {
        node [white] {$\scriptstyle \left\{\right\}$}
          child {
            node [white] {$\scriptstyle \left\{3\right\}$}
          }
          child {
            node [white] {$\scriptstyle \left\{4\right\}$}
          }
          child {
            node [red] {$\scriptstyle \left\{\right\}$}
              child {
                node [white] {$\scriptstyle \left\{5\right\}$}
                  child {
                    node [white] {$\scriptstyle \left\{6\right\}$}
                  }
              }
              child {
                node [white] {$\scriptstyle \left\{7\right\}$}
              }
          }
      }
      child [missing]
      child {
        node [white] {$\scriptstyle \left\{8\right\}$}
          child {
            node [white] {$\scriptstyle \left\{9\right\}$}
          }
      }
  };
\end{tikzpicture}
\\
\frac{\mbox{[3456789]} \mbox{[567]}}{\mbox{[123456789]} \mbox{[12]}
\mbox{[34567]} \mbox{[3]} \mbox{[4]} \mbox{[56]} \mbox{[6]} \mbox{[7]}
\mbox{[89]} \mbox{[9]}}
\end{gather*}

\section{Remarkable sub-operads}

As set-operads, there are many interesting and already known sub-operads of
$\BWTS$. On $\opGR$ trees, they correspond to selecting some operations
inside the four possible ones.
We shall see some examples and describe how they can be seen in terms of trees
in $\BWTS$.

\subsection{The operad $\BWT$}

Let us consider the sub-set-operad $\BWT$ of red and white trees generated
by
\begin{equation}
\left\{\,\geno12, \geng12, \gend12\,\right\}.
\end{equation}

We then have:

\begin{theorem}
\label{bwt}
The set sub-operad $\BWT$ of $\BWTS$ generated by the three previous trees has
as elements the red and white trees with no empty nodes.
\end{theorem}

\Proof
Given the product rules of $\BWTS$ (see Subsection~\ref{sub-op-rw}),
one easily checks that all trees belonging to $\BWT$ have no empty nodes.

Conversely, using the same technique as in the proof of
Proposition~\ref{prop-op4-bw}, one checks that any tree with no empty nodes
can be obtained as a composition of strictly smaller such trees.
\qed

So the cardinalities of this operad is Sequence~A200757
of~\cite{Sloane}, whose first elements are
\begin{equation}
  1, 3, 13, 68, 395, 2450, 15892, 106489, 731379, 5121392, 36425796, 262425982, \dots 
\end{equation}
with ordinary generating series $F=F_{\BWT}$ satisfying
\begin{equation}
  \label{eq_sg_T}
  F= (x+F)/(1-(x+F)) - F/(1-F),
\end{equation}
which is therefore algebraic.


\begin{note}
  For the corresponding unlabelled trees, we get the set of rooted
  trees with multiple dots, that correspond to Sequence~A036249
  of~\cite{Sloane}.  Their first numbers of elements are
  \begin{equation}
    1,2,5,13,37,108, 332, 1042, 3360, 11019, 36722, 123875, 422449, 1453553, \dots
  \end{equation}
\end{note}

\subsection{The operad $\BW$}

One can consider the subset $\BW$ of $\BWT$ of the subset $\BWTS$ generated by
\begin{equation}
\left\{\,\geng12, \gend12\,\right\}.
\end{equation}
We then have:
\begin{theorem}
\label{bw}
The set sub-operad $\BW$ of $\BWT$ generated by the two previous trees has
as elements the red and white trees with no empty nodes and no multiple labels
in their nodes. These trees are in immediate bijection with recursively
labelled rooted trees.
\end{theorem}

\Proof
The proof is the same as in Theorem~\ref{bwt}.
\qed

This set of trees corresponds to the set of trees that are in bijection with
the $\opGR$ trees only containing $\gautrid\,$ and $\,\droittrid\,$ in their
internal nodes. The cardinalities of this operad is Sequence~A006013
of~\cite{Sloane}, whose first elements are
\begin{equation}
  1, 2, 7, 30, 143, 728, 3876, 21318, 120175, 690690, 4032015, 23841480, \dots 
\end{equation}
with ordinary generating series $F=F_{\BW}$ satisfying
\begin{equation}
  \label{eq_sg_D}
  F= x/(1-F)^2, 
\end{equation}
which is therefore algebraic.

This corresponds to the set-operad based on non-crossing
trees~\cite{leroux,CHNT}.

\subsection{The operad $\BWS$}

There is another trickier way to study the sub-set-operads of $\BWTS$.
Let us consider the sub-set-operad $\BWS$ of red and white trees generated
by
\begin{equation}
\left\{\,\geng21, \geng12, \gens12\,\right\}.
\end{equation}

We then have:

\begin{theorem}
  \label{bws}
  The set sub-operad $\BWS$ of $\BWTS$ generated by the three previous
  trees has as elements the red and white trees with no multiple
  labels and that avoid the two following patterns: no labelled node
  can have a red child, and no white node can have two red children.
\end{theorem}

\Proof
Let us first show that the $\circ_i$ are internal on this set.
Let us consider the different cases appearing when computing
$T=T_1\circ_x T_2$, with $T_1$ and $T_2$ in $\BWS$.
If the root of $T_2$ is not red, $T$ also belongs to $\BWS$.
If the root of $T_2$ is red, and if $z$ is not a leaf, then $z$, being
labelled by $x$, cannot have a red child. So $T$ has now a white empty node
(replacing $z$) with only one red child. So $T$ also belongs to $\BWS$.

If the root of $T_2$ is red and $z$ is a leaf, then the parent of $z$ (if
$z$ has no parent, $T=T_2$ and the statement is trivial) can be whatever node,
it only gets white children, which does not fall in any forbidden pattern.

Now, given such a tree, one easily sees that it can be obtained as a
composition of strictly smaller such trees.
\qed

As a species, $\FSp = \FSp_\BWS$ is the solution of
\begin{align*}
  \FSp &= \BSp + \WSp         \\
  \BSp &= \SetSp_{\geq2}(\WSp)   \\
  \WSp &= \BSp\cdot \SetSp_{\geq1}(\WSp) + \ZSp\cdot\SetSp(\WSp)
\end{align*}

Here are the recursively labelled trees in $\BWTS$ of size 3:
\[\left[\ \begin{tikzpicture}
\node [root] {$\scriptstyle \left\{1\right\}$}
  child {
    node [white] {$\scriptstyle \left\{2\right\}$}
      child {
        node [white] {$\scriptstyle \left\{3\right\}$}
      }
  };
\end{tikzpicture}\,\begin{tikzpicture}
\node [root] {$\scriptstyle \left\{1\right\}$}
  child {
    node [white] {$\scriptstyle \left\{2\right\}$}
  }
  child {
    node [white] {$\scriptstyle \left\{3\right\}$}
  };
\end{tikzpicture}\,\begin{tikzpicture}
\node [root] {$\scriptstyle \left\{\right\}$}
  child {
    node [white] {$\scriptstyle \left\{3\right\}$}
  }
  child {
    node [red] {$\scriptstyle \left\{\right\}$}
      child {
        node [white] {$\scriptstyle \left\{2\right\}$}
      }
      child {
        node [white] {$\scriptstyle \left\{1\right\}$}
      }
  };
\end{tikzpicture}\,\begin{tikzpicture}
\node [red] {$\scriptstyle \left\{\right\}$}
  child {
    node [white] {$\scriptstyle \left\{1\right\}$}
  }
  child {
    node [white] {$\scriptstyle \left\{2\right\}$}
      child {
        node [white] {$\scriptstyle \left\{3\right\}$}
      }
  };
\end{tikzpicture}\,\begin{tikzpicture}
\node [red] {$\scriptstyle \left\{\right\}$}
  child {
    node [white] {$\scriptstyle \left\{2\right\}$}
  }
  child {
    node [white] {$\scriptstyle \left\{1\right\}$}
  }
  child {
    node [white] {$\scriptstyle \left\{3\right\}$}
  };
\end{tikzpicture}\ \right]\]

Here are those of size $4$:
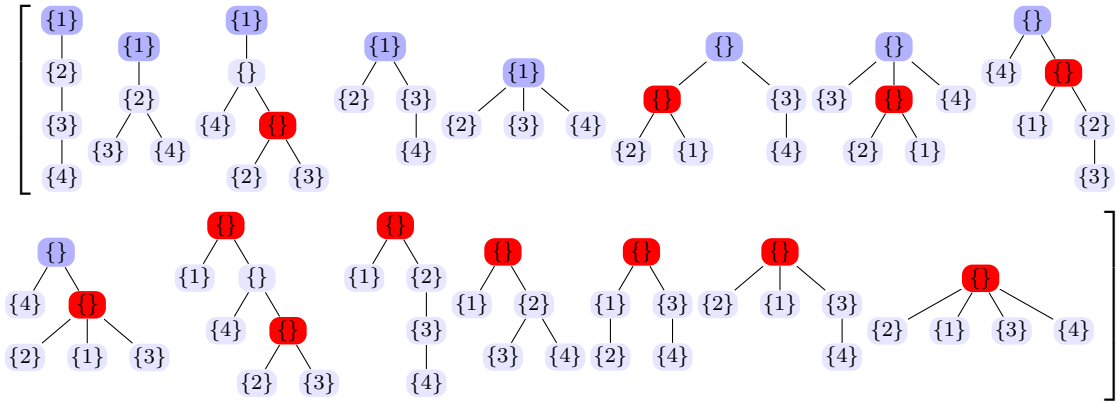
\[
\left[\ \begin{tikzpicture}
\node [root] {$\scriptstyle \left\{1\right\}$}
  child {
    node [white] {$\scriptstyle \left\{2\right\}$}
      child {
        node [white] {$\scriptstyle \left\{3\right\}$}
          child {
            node [white] {$\scriptstyle \left\{4\right\}$}
          }
      }
  };
\end{tikzpicture}\,\begin{tikzpicture}
\node [root] {$\scriptstyle \left\{1\right\}$}
  child {
    node [white] {$\scriptstyle \left\{2\right\}$}
      child {
        node [white] {$\scriptstyle \left\{3\right\}$}
      }
      child {
        node [white] {$\scriptstyle \left\{4\right\}$}
      }
  };
\end{tikzpicture}\,\begin{tikzpicture}
\node [root] {$\scriptstyle \left\{1\right\}$}
  child {
    node [white] {$\scriptstyle \left\{\right\}$}
      child {
        node [white] {$\scriptstyle \left\{4\right\}$}
      }
      child {
        node [red] {$\scriptstyle \left\{\right\}$}
          child {
            node [white] {$\scriptstyle \left\{2\right\}$}
          }
          child {
            node [white] {$\scriptstyle \left\{3\right\}$}
          }
      }
  };
\end{tikzpicture}\,\begin{tikzpicture}
\node [root] {$\scriptstyle \left\{1\right\}$}
  child {
    node [white] {$\scriptstyle \left\{2\right\}$}
  }
  child {
    node [white] {$\scriptstyle \left\{3\right\}$}
      child {
        node [white] {$\scriptstyle \left\{4\right\}$}
      }
  };
\end{tikzpicture}\,\begin{tikzpicture}
\node [root] {$\scriptstyle \left\{1\right\}$}
  child {
    node [white] {$\scriptstyle \left\{2\right\}$}
  }
  child {
    node [white] {$\scriptstyle \left\{3\right\}$}
  }
  child {
    node [white] {$\scriptstyle \left\{4\right\}$}
  };
\end{tikzpicture}\,\begin{tikzpicture}
\node [root] {$\scriptstyle \left\{\right\}$}
  child {
    node [red] {$\scriptstyle \left\{\right\}$}
      child {
        node [white] {$\scriptstyle \left\{2\right\}$}
      }
      child {
        node [white] {$\scriptstyle \left\{1\right\}$}
      }
  }
  child [missing]
  child {
    node [white] {$\scriptstyle \left\{3\right\}$}
      child {
        node [white] {$\scriptstyle \left\{4\right\}$}
      }
  };
\end{tikzpicture}\,\begin{tikzpicture}
\node [root] {$\scriptstyle \left\{\right\}$}
  child {
    node [white] {$\scriptstyle \left\{3\right\}$}
  }
  child {
    node [red] {$\scriptstyle \left\{\right\}$}
      child {
        node [white] {$\scriptstyle \left\{2\right\}$}
      }
      child {
        node [white] {$\scriptstyle \left\{1\right\}$}
      }
  }
  child {
    node [white] {$\scriptstyle \left\{4\right\}$}
  };
\end{tikzpicture}\,\begin{tikzpicture}
\node [root] {$\scriptstyle \left\{\right\}$}
  child {
    node [white] {$\scriptstyle \left\{4\right\}$}
  }
  child {
    node [red] {$\scriptstyle \left\{\right\}$}
      child {
        node [white] {$\scriptstyle \left\{1\right\}$}
      }
      child {
        node [white] {$\scriptstyle \left\{2\right\}$}
          child {
            node [white] {$\scriptstyle \left\{3\right\}$}
          }
      }
  };
\end{tikzpicture}\right.
\]
\[\left.
\begin{tikzpicture}
\node [root] {$\scriptstyle \left\{\right\}$}
  child {
    node [white] {$\scriptstyle \left\{4\right\}$}
  }
  child {
    node [red] {$\scriptstyle \left\{\right\}$}
      child {
        node [white] {$\scriptstyle \left\{2\right\}$}
      }
      child {
        node [white] {$\scriptstyle \left\{1\right\}$}
      }
      child {
        node [white] {$\scriptstyle \left\{3\right\}$}
      }
  };
\end{tikzpicture}\,\begin{tikzpicture}
\node [red] {$\scriptstyle \left\{\right\}$}
  child {
    node [white] {$\scriptstyle \left\{1\right\}$}
  }
  child {
    node [white] {$\scriptstyle \left\{\right\}$}
      child {
        node [white] {$\scriptstyle \left\{4\right\}$}
      }
      child {
        node [red] {$\scriptstyle \left\{\right\}$}
          child {
            node [white] {$\scriptstyle \left\{2\right\}$}
          }
          child {
            node [white] {$\scriptstyle \left\{3\right\}$}
          }
      }
  };
\end{tikzpicture}\,\begin{tikzpicture}
\node [red] {$\scriptstyle \left\{\right\}$}
  child {
    node [white] {$\scriptstyle \left\{1\right\}$}
  }
  child {
    node [white] {$\scriptstyle \left\{2\right\}$}
      child {
        node [white] {$\scriptstyle \left\{3\right\}$}
          child {
            node [white] {$\scriptstyle \left\{4\right\}$}
          }
      }
  };
\end{tikzpicture}\,\begin{tikzpicture}
\node [red] {$\scriptstyle \left\{\right\}$}
  child {
    node [white] {$\scriptstyle \left\{1\right\}$}
  }
  child {
    node [white] {$\scriptstyle \left\{2\right\}$}
      child {
        node [white] {$\scriptstyle \left\{3\right\}$}
      }
      child {
        node [white] {$\scriptstyle \left\{4\right\}$}
      }
  };
\end{tikzpicture}\,\begin{tikzpicture}
\node [red] {$\scriptstyle \left\{\right\}$}
  child {
    node [white] {$\scriptstyle \left\{1\right\}$}
      child {
        node [white] {$\scriptstyle \left\{2\right\}$}
      }
  }
  child {
    node [white] {$\scriptstyle \left\{3\right\}$}
      child {
        node [white] {$\scriptstyle \left\{4\right\}$}
      }
  };
\end{tikzpicture}\,\begin{tikzpicture}
\node [red] {$\scriptstyle \left\{\right\}$}
  child {
    node [white] {$\scriptstyle \left\{2\right\}$}
  }
  child {
    node [white] {$\scriptstyle \left\{1\right\}$}
  }
  child {
    node [white] {$\scriptstyle \left\{3\right\}$}
      child {
        node [white] {$\scriptstyle \left\{4\right\}$}
      }
  };
\end{tikzpicture}\,\begin{tikzpicture}
\node [red] {$\scriptstyle \left\{\right\}$}
  child {
    node [white] {$\scriptstyle \left\{2\right\}$}
  }
  child {
    node [white] {$\scriptstyle \left\{1\right\}$}
  }
  child {
    node [white] {$\scriptstyle \left\{3\right\}$}
  }
  child {
    node [white] {$\scriptstyle \left\{4\right\}$}
  };
\end{tikzpicture}\ \right]
\]

The trees are in bijection with the $\opGR$ trees only containing
$\gautrid$, $\droittrid$, and $\odot$ in their internal nodes. So

\begin{corollary}
The set-operad on $\BWS$ is isomorphic to the set-operad on the three
tridendriform operations $\gautrid$, $\odot$, and $\droittrid$.
\end{corollary}

The cardinalities of this operad is Sequence~A121873
of~\cite{Sloane}, whose first elements are
\begin{equation}
  1, 3, 14, 80, 510, 3479, 24848, 183465, 1389090, 10726452, 84150858, \dots 
\end{equation}
with ordinary generating series $F=F_{\BWS}$ satisfying
\begin{align*}
  \label{eq_sg_S}
  F &= R + W         \\ 
  R &= 1/(1-W) -1 -W  \\
  W &= R/(1-W)^2 - R + x/(1-W)^2,
\end{align*}
which is therefore algebraic.


This corresponds to the operad of noncrossing plants already defined by
Chapoton~\cite{Ch}.

This sequence also appears as Example (g) in \cite{Lod04}.





\section{Symmetric operads}

\label{symmop}

All results presented above have analogs in the world of symmetric operads.
Indeed, given either the $\opGR$ trees, or the red and white trees, or the
fractions, one can consider their pendant as symmetric operads.

\subsection{The symmetric operad on red and white trees}

In Section~\ref{sub-op-rw}, an operad structure was defined on the set of red
and white trees with arbitrary labeling. Relabeling by a permutation clearly
defines an action of the symmetric groups. Those tree forms a species which is
defined by Equation~\eqref{eq_FSp}.
\begin{theorem}
\label{thm-rw-oper}
The set $\BWTSL$ endowed with operations $\circ_x$ and the natural symmetric
groups actions has a symmetric set-operad structure. This operad will be
also be denoted by $\BWTSL$.
\end{theorem}
\begin{proof}
  The definition of the composition clearly commutes with relabeling.
\end{proof}
Let us investigate more closely the actions of the groups. The species
equation allows to compute the characters of the representations which we
prefer to encode as a symmetric function using Frobenius characteristic (see
e.g.~\cite{MacDo}). Using the notations from the later, Equation~\eqref{eq_FSp}
translate directly on Frobenius characteristic into the following equation:
\begin{equation}
  F = (E_1 \circ F - 1 - F) + (E_t - 1) (E_1\circ F)\,
\end{equation}
where $E_t = \sum_{i\geq 0} e_i t^i$ is the generating series of elementary
symmetric functions, and $\circ$ is the plethysm operation. The solution of
this equation can be computed inductively. Here are the first few terms
expanded on Schur functions:
$$s_{1}, s_{1,1} + 3s_{2}, 10s_{3} + 2s_{1,1,1} + 10s_{2,1}, 6s_{1,1,1,1} + 40s_{4} + 64s_{3,1} + 38s_{2,2} + 34s_{2,1,1}$$
We also need some information on the orbits:
\begin{proposition}\label{prop_orbits}
  Under the action of the symmetric groups on $\BWTSL$ there exists at least
  one recursively labeled tree in each orbit.
\end{proposition}
\begin{proof}
  It is sufficient to prove that there exists a recursive labeling for each
  unlabeled red and white tree. This is clear as such a labeling can be
  obtained by a recursive depth-first walk of a tree labeling for example the
  dots in the roots before the dots in the sub-trees.
\end{proof}
\begin{corollary}\label{cor_orbits}
  As a symmetric operad, $\BWTSL$ is generated by the trees:
  \begin{equation}
    \label{def_GSigma}
    G^\Sigma := \left\{\ \geng12\ \geno12\ \gens12\ \right\}.
  \end{equation}
\end{corollary}
\begin{proof}
  Note that the action of the transposition $(12)$ on $e:=\geng12$ is
  $f:=\gend12$. Therefore it is not necessary to put the later in the generators.
  By Proposition~\ref{prop-op4-bw}, the set of elements obtained by composition
  from $G^\Sigma\cap\{f\}$, without using the action of the symmetric groups
  is exactly the set of recursively labeled trees. Therefore using the action
  we get at least the set of the orbit of the recursively labeled trees, that
  is all labeled trees.
\end{proof}
The main goal of this section is to show that the symmetric operad $\BWTSL$ is
isomorphic to the symmetric sub-operad of formal fraction generated by the
corresponding fraction.

\subsubsection{The symmetric operad on fractions}

In Section~\ref{sec_bij_RW_Frac}, we defined a map $\phi$ from red and white
tree to formal fraction; See Equation~\eqref{eq_bij_RW_Frac}. We showed that
this map restricted to recursively labeled trees is an isomorphism of
(non-symmetric) operads from $\BWTS$ to $\opFF$. It is also an
isomorphism of symmetric operads:
\begin{theorem}
  The map $\Phi$ is an isomorphism of symmetric operads from $\BWTSL$ to
  the sub-operad of $\SYMOP\opFF$ generated by the fractions:
  \begin{equation}
    F_{\droittrid}:=\frac{1}{[1][12]},\quad
    F_{\miltrid}:=\frac{1}{[12]},\quad
    F_{\odot}:=\frac{1}{[1][2]}\,.
  \end{equation}
\end{theorem}
Note: the action of the transposition $(1,2)$ on $F_{\droittrid}$ is
$F_{\gautrid}:=\frac{1}{[2][12]}$.
\begin{proof}
  We already proved (Lemma~\ref{lemma_Phi_injective}) that as a set map $\Phi$
  is injective from $\BWTSL$ to $\opFF$. From its definition it is also clear
  that $\Phi$ commutes with the action of the symmetric group. Now we know
  that on recursively labeled trees, $\Phi$ commute with compositions
  $\circ_i$ (Theorem~\ref{theorem_main}). Since there is a recursively labeled
  tree in every orbit, using the action we get that $\Phi$ commute with
  compositions $\circ_i$ on all red and white trees. Therefore $\Phi$ is an
  injective symmetric operad morphism from $\BWTSL$ to $\SYMOP\opFF$. Let us
  consider its image $\Phi(\BWTSL)$. It is generated by $\left\{\Phi(g)\mid
    g\in G\right\}$, where $G$ is any generating set of $\BWTSL$. The theorem
  is then obtained using the set $\G^\Sigma$ of Equation~\eqref{def_GSigma}.
\end{proof}


\subsection{The symmetric sub-operads}

Let us consider a subset of the generators of $\BWTS$ and the symmetric and
non symmetric operads generated by this subset. Using the same argument as in
Proposition~\ref{prop_orbits} and Corollary \ref{cor_orbits}, one sees that
the underlying set of the symmetric one is given by the orbits of the
symmetric group on the underlying set of the non-symmetric one. This amounts
to consider the set of the same red and white trees with no restrictions on
the labels.

\subsubsection{The symmetric sub-operad $\BWTL$}

Let us consider the symmetric analog of $\BWT$. When labelling red
and white trees with no empty nodes with different integers from $1$
to $n$ without any other constraint, one gets Sequence~A048802
of~\cite{Sloane} whose first number of elements are
\begin{equation}
  1, 3, 16, 133, 1521, 22184, \dots 
\end{equation}
with exponential generating function $F=F_{\BWTL}$ satisfying
\begin{equation}
  \label{eq_sg_TL}
  F = (\exp(x)-1) \exp(F).
\end{equation}

\subsubsection{The symmetric sub-operad $\BWL$}

Let us consider the symmetric analog of $\BW$.  When labelling red and white
trees with no empty or multiple nodes with different integers from $1$ to $n$
without any other constraint, one gets Sequence~A000169 of~\cite{Sloane} whose
value is $n^{n-1}$ and first number of elements are
\begin{equation}
1, 2, 9, 64, 625, 7776, 117649, 2097152, 43046721, 1000000000, 25937424601 
\end{equation}
with exponential generating function $F=F_{\BWL}$ satisfying
\begin{equation}
  \label{eq_sg_L}
  F = x \exp(F).
\end{equation}

This is isomorphic to the set-operad $\NAP$\footnote{standing for Non-Associative Permutative} \cite{LivernetNAP}, because it is
generated by one generator which satisfies the same relations as the
generator of the $\NAP$ operad and the dimensions are the same.

\subsubsection{The symmetric sub-operad $\BWSL$}

Let us consider the symmetric analog of $\BWS$. When labelling the
corresponding red and white trees with different integers from $1$ to
$n$ without any other constraint, one gets Sequence~A048172
of~\cite{Sloane} whose first terms are
\begin{equation}
  1, 3, 19, 195, 2791, 51303, 1152019, 30564075, 935494831, 32447734143, \dots 
\end{equation}
with exponential generating function $F=F_{\BWSL}$ satisfying
\begin{align*}
  \label{eq_sg_SL}
  F &= R + W         ,\\
  R &= \exp(W) -1 -W   ,\\
  W &= R (\exp(W) -1) + x \exp(W).
\end{align*}

Recalling the constraints described in Theorem \ref{bws}, one can
obtain this system of equations as follows. Here $R$ (resp. $W$)
denotes the generating series of trees with a red root (resp with a
white root). The second equation says that a red root has only white
sons. The last equation says that a white root is either empty and has
exactly one red son, or has a label and an arbitrary set of white
sons.

\begin{note}
  This operad is isomorphic to the operads of shrubs, which was
  defined in \cite{shrub,shrubbij} by the very same closure as this
  operad. The combinatorial objects called shrubs are directed graphs
  with levels in $\NN$, where arrows go down by one level, satisfying
  some forbidden pattern conditions. This is rather different at first
  sight from red and white trees with the given constraints.

  Using the operad structure, once the components of arity $2$ have
  been matched, one can easily define by induction a bijection between
  shrubs and red and white trees with the given constraints, that
  gives an isomorphism of operads between the operad of shrubs and
  $\BWSL$.
\end{note}









\section{More general operads on more generators}

Recall that there is a natural morphism from the dendriform operad to
the tridendriform operad sending $\succ$ to $\circ\ + \succ$ and
$\prec$ to itself (the other convention is also possible). It is therefore very natural to add the
counterpart in $\opMould1$ of $\circ\ + \succ$ and $\circ\ + \prec$ to
the generating set of $\opFF_4$. This leads to several interesting
operads which do not live in formal fractions, but rather in a
slightly more general operad of formal fractions with monomials, which
is defined as a Hadamard product. In this section, we present the
combinatorial properties of these operads, omitting most of the proofs
as they are either direct consequences of the previous results or
derived by very similar reasoning.  \bigskip

Recall that the generators of the tridendriform operad are realized as
moulds by \eqref{def_phi1_sup}, \eqref{def_phi1_inf},
\eqref{def_phi1_circ} and \eqref{def_phi1_sum}. With this convention
the two new generators are realized as:
\begin{equation}
\circ\ + \succ\ \mapsto\ \frac{u_1}{(u_1-1)(u_1u_2-1)},   \quad
\circ\ + \prec\ \mapsto\ \frac{u_2}{(u_2-1)(u_1u_2-1)},\\
\end{equation}
We consider here the sub-operad generated by the four fractions together with
these two new.  To be able to deal with this operad using formal fractions we
need to generalize them a little using an operad on monomials: let $\opMon(n)$
be the set of monomials in $u_1, \dots, u_n$. The composition defined, for
$F\in\opMon(m)$ and $G\in\opMon(n)$, by
\begin{equation}
  \label{def.composition.mon}
  F \circ_i G :=
  F(u_1, \dots, x_{i-1}, P_{i,n}, u_{i+1},\dots,u_{m+n-1})\,
  G(u_i, \dots, u_{i+n-1})\,
\end{equation}
where $P_{i,n}=u_i u_{i+1}\dots u_{i+n-1}$, endows $\opMon$ with a
structure of a set-operad. Together with the action of the symmetric
group, it becomes a symmetric set-operad denoted by $\SYMOP{\opMon}$.
This operad has already appeared in~\cite{Lod10}. One denotes by $\opMFF:=\opMon\times\opFF$ the Hadamard
product of these set-operads. Elements of $\opMFF(n)$ are pairs
$(m,f) \in \opMon(n)\times\opFF(n)$, the composition being defined
componentwise. We denote such an element by putting the monomial in
the numerator of the formal fraction.  Then it is clear from the
definition of the composition in $\opMould1$
(by Eq. \eqref{def.composition.mould1}) that the morphism $\phi_1$ from
$\opFF$ to $\opMould1$ extends to $\opMFF$:
\begin{proposition}
  The map $\phi_{\mathcal{M}}:(m,f) \mapsto\ m\,\phi_1(f)$ is an injective
  morphism of set-operads from $\opMFF$ to $\opMould1$. It is also an
  injective morphism of symmetric set-operads with the symmetric version of
  those operads.
\end{proposition}

\subsection{The $6$-generator operad and more}
We now consider the sub-operad $\opFF_6$ of $\opMFF$, generated by the four
generator of $\opFF_4$ together with
\begin{equation}
  F_{\gg}:=\frac{u_1}{[1][12]},\quad
  F_{\ll}:=\frac{u_2}{[2][12]}\,.
\end{equation}
The generators satisfies the following $16$ relations:
\begin{equation}\label{relations16}
\begin{aligned}
  \treeLft{\odot}{\odot}\ &=\ \treeRgt{\odot}{\odot}
  &
  \treeLft{\circ}{\circ}\ &=\ \treeRgt{\circ}{\circ}
  &
  \treeLft{\circ}{\prec}\ &=\ \treeRgt{\circ}{\succ}
  &
  \treeLft{\circ}{ \ll }\ &=\ \treeRgt{\circ}{ \gg }
  \\[5mm]
  \treeLft{\prec}{\prec}\ &=\ \treeRgt{\prec}{\odot}
  &
  \treeLft{ \ll }{ \ll }\ &=\ \treeRgt{ \ll }{\odot}
  &
  \treeLft{\succ}{\odot}\ &=\ \treeRgt{\succ}{\succ}
  &
  \treeLft{ \gg }{\odot}\ &=\ \treeRgt{ \gg }{ \gg }
  \\[5mm]
  \treeLft{\prec}{\circ}\ &=\ \treeRgt{\circ}{\prec}
  &
  \treeLft{ \ll }{\circ}\ &=\ \treeRgt{\circ}{ \ll }
  &
  \treeLft{\circ}{\succ}\ &=\ \treeRgt{\succ}{\circ}
  &
  \treeLft{\circ}{ \gg }\ &=\ \treeRgt{ \gg }{\circ}
  \\[5mm]
  \treeLft{\prec}{\succ}\ &=\ \treeRgt{\succ}{\prec}
  &
  \treeLft{ \ll }{\succ}\ &=\ \treeRgt{\succ}{ \ll }
  &
  \treeLft{\prec}{ \gg }\ &=\ \treeRgt{ \gg }{\prec}
  &
  \treeLft{ \ll }{ \gg }\ &=\ \treeRgt{ \gg }{ \ll }
\end{aligned}
\end{equation}
and generates an operad whose cardinalities are
$$1, 6, 56, 640, 8158, 111258, 1588544, 23446248, 354855218,\dots$$
This can be computed by counting the canonical trees as in Section
\ref{subsection:canonical}. The system of equations for generating
series, analog to \eqref{systeme_gr4}, writes
\begin{equation}
  \label{systeme_gr6}
  \left\{
    \begin{aligned}
      f &= l + L + m + r + R + s + 1, \\
      l &= x \,f\,(l+L+m+r+R+1),  \\
      L &= x \,f\,(l+L+m+r+R+1),  \\
      r &= x \,f\,(R+s+1),      \\
      R &= x \,f\,(l+s+1),      \\
      m &= x \,f\,(s+1),      \\
      s &= x \,f\,(l+L+m+r+R+1).  \\
    \end{aligned}
  \right.
\end{equation}
where $f$, $l$, $L$, $m$, $r$, $R$ and $s$ denotes respectively the
generating series of all canonical trees, all canonical trees having
$\prec$, $\ll$ $\circ$, $\succ$, $\gg$ and $\odot$ as their root.

Eliminating all the variables but $f$ and $x$, and setting $F=xf$
gives the following algebraic equation for the generating series
\begin{equation}
\label{solution_systeme_gr6}
F = x + 3x F + (3+x) {F}^2 +(1-x) {F}^3 - {F}^4\,,
\end{equation}
which can be rewritten as
\begin{equation}
x = F\frac{1 - 3F - F^2 + F^3}{1 + 3F + F^2 - F^3}\,.
\end{equation}

To show that there are no other relations, we need to use a generalized red
and white trees operad. It is defined as the set of trees where one can put a
dot on the edges between two white nodes. The associated fraction is the
fraction associated to non-dotted tree times the product on each label of the
variable associated to it to the power the number of dotted edges on the path
from the root. This clearly defines an injection $\phi_\mathcal{M}$ from dotted
red and white trees to $\opMFF$.

Here are two examples:
\[\begin{tikzpicture}
\node [root] {$\scriptstyle \left\{8\right\}$}
  child {
    node [white] {$\scriptstyle \left\{3, 4\right\}$}
      child {
        node [red] {$\scriptstyle \left\{\right\}$}
          child {
            node [white] {$\scriptstyle \left\{2\right\}$}
          }
          child {
            node [white] {$\scriptstyle \left\{1\right\}$}
          }
      }
    edge from parent [-*, thick]
  }
  child [missing]
  child {
    node [white] {$\scriptstyle \left\{\right\}$}
      child {
        node [white] {$\scriptstyle \left\{7\right\}$}
        edge from parent [-*, thick]
      }
      child {
        node [red] {$\scriptstyle \left\{\right\}$}
          child {
            node [white] {$\scriptstyle \left\{6\right\}$}
          }
          child {
            node [white] {$\scriptstyle \left\{5\right\}$}
          }
      }
  };
\end{tikzpicture}\qquad
\frac{u_1u_2u_3u_4u_7\mbox{[12]} \mbox{[56]}}{\mbox{[1]} \mbox{[2]} \mbox{[1234]}
\mbox{[5]} \mbox{[6]} \mbox{[7]} \mbox{[567]} \mbox{[12345678]}}\]

\[\begin{tikzpicture}
\node [root] {$\scriptstyle \left\{8\right\}$}
  child {
    node [white] {$\scriptstyle \left\{3, 4\right\}$}
      child {
        node [red] {$\scriptstyle \left\{\right\}$}
          child {
            node [white] {$\scriptstyle \left\{2\right\}$}
          }
          child {
            node [white] {$\scriptstyle \left\{1\right\}$}
          }
      }
  }
  child [missing]
  child {
    node [white] {$\scriptstyle \left\{\right\}$}
      child {
        node [white] {$\scriptstyle \left\{7\right\}$}
        edge from parent [-*, thick]
      }
      child {
        node [red] {$\scriptstyle \left\{\right\}$}
          child {
            node [white] {$\scriptstyle \left\{6\right\}$}
          }
          child {
            node [white] {$\scriptstyle \left\{5\right\}$}
          }
      }
    edge from parent [-*, thick]
  };
\end{tikzpicture}\qquad
\frac{u_5u_6u_7^2\mbox{[12]} \mbox{[56]}}{\mbox{[1]} \mbox{[2]} \mbox{[1234]}
\mbox{[5]} \mbox{[6]} \mbox{[7]} \mbox{[567]} \mbox{[12345678]}}\]

The extension of the rules for the operad composition is
straightforward, except for rule (R3) which should be modified as
follows: let us consider two trees $T_1$ and $T_2$ such that $x$ is
the only label of a leaf $z$. Suppose moreover that there is a dotted
edge from $z'$ to $z$.  The composition $T_1 \circ_x T_2$ is then
defined as the tree obtained by remove the leaf $z$ and putting the
children of the root of $T_2$ as new dotted children of $z'$. On can
easily check that this defines an operad on dotted trees such
that $\phi_\mathcal{M}$ is a morphism to $\opMFF$. For example, the
following equality
$$
\begin{tikzpicture}
\node [root]{$\scriptstyle \left\{1\right\}$}
  child {
    node [white] {$\scriptstyle \left\{2\right\}$}
    edge from parent [-*, thick]
  }
  child {
    node [white] {$\scriptstyle \left\{3,4\right\}$}
  };
\end{tikzpicture}\circ_2\blackTree
\ =\ 
\begin{tikzpicture}
\node [root]{$\scriptstyle \left\{1\right\}$}
  child {
    node [white] {$\scriptstyle \left\{2\right\}$}
    edge from parent [-*, thick]
  }
  child {
    node [white] {$\scriptstyle \left\{4\right\}$}
      child {
        node [white] {$\scriptstyle \left\{3\right\}$}
      }
    edge from parent [-*, thick]
  }
  child {
    node [white] {$\scriptstyle \left\{5,6\right\}$}
  };
\end{tikzpicture}
$$
is mapped to the following formal fraction composition
$$
\frac{u_{2}}{\mbox{[1234]} \mbox{[2]} \mbox{[34]}} 
\circ_2
\frac{1}{\mbox{[1]} \mbox{[23]} \mbox{[2]}}
\ =\ 
\frac{u_{2} u_{3} u_{4}}{\mbox{[123456]} \mbox{[2]} \mbox{[34]} \mbox{[3]} \mbox{[56]}}
$$

The number of those dotted trees can be obtained by counting the number of
white-white edges in the undotted tree compositions. This can be done by
refining Equation \eqref{eq_sg_TS_first} using a variable $t$ to record those
edges. Then we can show that the generating series verifies the following
equations:
\begin{multline}\label{rwTreesAlgEq:t}
F = -{\left(t - 1\right)} F^{4} - {\left(t - 1\right)} F^{3} x - {\left(t -
    3\right)} F^{2} x - {\left({\left(t - 1\right)}^{2} - 2\right)} F^{3} \\
+ {\left(2 \, t - 1\right)} F^{2} + 3 \, F x + x
\end{multline}
\begin{equation}
x = F\frac{{\left(t - 1\right)} F^{3} + {\left(t^{2} - 2 \, t - 1\right)} F^{2}
- {\left(2 \, t - 1\right)} F + 1}{\left(-{\left(t - 1\right)} F^{3} -
{\left(t - 3\right)} F^{2} + 3 \, F + 1\right)}
\end{equation}
Here are the first generating polynomials:
\begin{equation}
  \begin{aligned}
    1\\
    2\,t + 2\\
    7\,t^2 + 11\,t + 6\\
    30\,t^3 + 65\,t^2 + 59\,t + 22\\
    143\,t^4 + 397\,t^3 + 492\,t^2 + 318\,t + 90\\
    728\,t^5 + 2471\,t^4 + 3857\,t^3 + 3430\,t^2 + 1728\,t + 394
  \end{aligned}
\end{equation}
One can check that substituting $t=2$ in \eqref{rwTreesAlgEq:t} gives back
\eqref{solution_systeme_gr6}. This can be used to show the following theorem:
\begin{theorem}
  The set-operad $\opFF_6$, the operad presented by Equation
  \eqref{relations16} and the operad of recursively labelled dotted red and
  white trees are isomorphic.
\end{theorem}
Substituting $t=0$ is the preceding generating series gives the cardinalities
of the operad of red and white trees with no white-white edges. It is
isomorphic to the operad generated by $\{\circ,\odot\}$ that is the operad
with two associative operations and no other relations. The cardinalities are
known as large Schroeder number (Sloane's sequence A006318). The leading
coefficient corresponds to the operad $\BW$.
\bigskip

Finally these $t$-parametrized generating series suggests the
existence of a family of operads $\BWTS_k$ indexed by any integer
$k\in\NN$ whose generating series of dimensions is given by the
solution of Equation \eqref{rwTreesAlgEq:t} with $t=k$. Such an operad
can be defined as the extension of the red and white trees where
instead of putting dots of one color on white-white edges one can put
dots of $k$ possible colors.

\subsection{Symmetric counterparts}
All the operads considered in the previous subsection have their symmetric
counterparts. As a species, the $t$ colored red and white trees are given by
the following equation system:
\begin{equation}
  \begin{aligned}
    \FSp &=\WSp + \RSp \\
    \BSp &= \SetSp_{\geq2}(\WSp)   \\
    \WSp &= \SetSp_{\geq1}(\ZSp)\cdot\SetSp(t\WSp + \RSp) +
    \RSp\cdot\SetSp_{\geq1}(t\WSp) +
    \SetSp_{\geq2}(\RSp)\cdot\SetSp(t\WSp)
  \end{aligned}
\end{equation}
From the preceding system, one can of course extract equations for the
exponential generating series $\FSp(x)$. Here are the result after eliminating
$\BSp$ and simplifying:
\begin{gather}
  \begin{aligned}
    \FSp &= \exp(\WSp) - 1 \\
    \WSp &= \exp(x + t\WSp + \exp(\WSp) -\WSp - 1) +\WSp - \exp(t\WSp) -
    \exp(\WSp) + 1
  \end{aligned}
\end{gather}
The coefficients in $x$ count the number of arbitrary labeled red and white trees, with white-white edges colored with $t$ possible
color. Here are the first values:
\begin{equation}
  \begin{aligned}
    1\\
    2\,t + 2\\
    9\,t^2 + 15\,t + 8\\
    64\,t^3 + 156\,t^2 + 144\,t + 52\\
    625\,t^4 + 2050\,t^3 + 2675\,t^2 + 1730\,t + 472\\
    7776\,t^5 + 32430\,t^4 + 55000\,t^3 + 50310\,t^2 + 25108\,t + 5504 
  \end{aligned}
\end{equation}
Dotted trees appear when $t=2$, giving the following cardinalities:
\begin{equation}
  1, 6, 74, 1476, 41032, 1464672, 63865328, 3290120832, 195537380704
\end{equation}
\begin{theorem}
  The symmetric set-operad $\SYMOP\opFF_6$ and the symmetric operad of dotted red and
  white trees are isomorphic.
\end{theorem}

As in the non symmetric case, the constant coefficients (Sloane's A006351) are
the cardinalities of the operad of red and white trees with no white-white
edges. It is isomorphic to the operad generated by $\{\circ,\odot\}$ that is
the operad with two associative and commutative operations and no other
relations. Indeed, this operad can be naturally encoded by series-parallel
networks with $n$ labeled edges. %
\def\NSp{\mathcal{N}}%
\def\PSp{\mathcal{P}}%
\def\SSp{\mathcal{S}}%
Recall that series-parallel networks are defined as a species by
\begin{equation}
  \begin{aligned}
    \NSp &= \ZSp + \SSp + \PSp \\
    \SSp &= \SetSp_2(\ZSp + \PSp)\\
    \PSp &= \SetSp_2(\ZSp + \SSp)
  \end{aligned}
\end{equation}
of course as a species $\SSp=\PSp$. The bijection $\phi$ with red and white
trees goes inductively as follows:
\begin{itemize}
\item The singleton $\ZSp$ is the identity of the operad and corresponds to
  the tree consisting only of a leaf labeled $1$;
\item A series network with edges labeled by $a,b,\dots$ and with parallel
  sub-networks $A,B,\dots$ corresponds to a white node labeled by the set
  $\{a,b,c\dots\}$ and with red sub-trees $\phi(A),\phi(B),\dots$.
\item A parallel network with edges labelled by $a,b,\dots$ and with series
  sub-networks $A,B,\dots$ corresponds to a red rooted sub-trees (of size
  $\geq2$) $\phi(A),\phi(B),\dots$, and leaves labeled by $a,b\dots$.
\end{itemize}
This bijection clearly commutes with the action of the symmetric groups by
relabeling, showing that red and white trees and series-parallel network are
isomorphic as species. One can then check that the morphism is actually a
morphism of operads.

Finally, one can remark that, as in the non-symmetric case, the
leading coefficient corresponds to the operad $\BWL$.


\end{document}